\documentclass[11pt]{amsproc}
\usepackage[mathcal]{euscript}
\usepackage{amsmath,amssymb,amsfonts,amsthm}
\usepackage{latexsym,bm,graphicx}
\usepackage{mathrsfs}
\usepackage[all]{xy}
\usepackage{extarrows}
\usepackage{enumitem}
\usepackage{setspace}
\usepackage[verbose,colorlinks=true,linktocpage=true,linkcolor=blue,citecolor=blue]{hyperref}
\usepackage{footnote}
\usepackage{tikz}
\usepackage[title]{appendix}
\usepackage{indentfirst}
\usepackage{tikz-cd}

\newtheorem{theorem}{Theorem}[section]
\newtheorem{lemm}[theorem]{Lemma}
\newtheorem{prop}[theorem]{Proposition}
\theoremstyle{definition}
\newtheorem{defi}[theorem]{Definition}

\newtheorem{coro}[theorem]{Corollary}
\theoremstyle{remark}
\newtheorem{remark}[theorem]{Remark}
\newtheorem*{clm*}{Claim}

\numberwithin{equation}{section}

\usetikzlibrary{graphs,matrix}
\usepgflibrary{arrows}

\begin{document}

\title[Affine Yangians as Limits of Quantum Toroidal Algebras]
{Affine Yangians as Limits of Quantum Toroidal Algebras}

\author[L. Bezerra, I. Kashuba, and H. Lin]{Luan Bezerra,  Iryna Kashuba, and Hongda Lin}

\address[bezerra.luan@gmail.com ]{\textsc{Luan Pereira Bezerra}: Instituto de Ci{\^e}ncias Exatas, Universidade Federal de Minas Gerais, Belo Horizonte, Brazil}

\address[kashuba@sustech.edu.cn]{\textsc{Iryna Kashuba}: Shenzhen International Center for Mathematics, Southern University of Science and Technology, Shenzhen, P.R. China}

\address[linhd@sustech.edu.cn]{\textsc{Hongda Lin}: Shenzhen International Center for Mathematics, Southern University of Science and Technology, Shenzhen, P.R. China}

\subjclass[2020]{17B37, 17B65, 17B67}

\keywords{Kac-Moody Lie algebras; Quantum toroidal algebras; Affine Yangians; PBW basis}
\begin{abstract}
We establish a degeneration isomorphism between quantum toroidal algebras
and untwisted affine Yangians, valid for all untwisted affine Kac-Moody Lie algebras.
Specifically, we prove that the affine Yangian $Y_\hbar(\mathfrak{g})$ is isomorphic, as a
$\mathbb{C}[\hbar]$-algebra, to the associated graded algebra of the quantum toroidal algebra
$U_\hbar(\mathfrak{g}^{\mathrm{tor}})$ with respect to a canonical filtration. This result
constitutes the affine analogue of Drinfeld's conjecture on the relationship between Yangians
and quantum loop algebras, previously established in the finite-dimensional setting by Gautam--Toledano Laredo
and by Guay--Ma.
As principal applications of this isomorphism, we derive two fundamental structural properties
of affine Yangians: a Poincar\'e--Birkhoff--Witt (PBW) basis for $Y_\hbar(\mathfrak{g})$ in all
untwisted affine types, and the identification of its classical limit as the universal enveloping
algebra $U(\mathfrak{g}[u])$ of the polynomial current Lie algebra. A key ingredient of independent
interest is our construction of a PBW basis for $U_\hbar(\mathfrak{g}^{\mathrm{tor}})$ itself,
which relies on a new torsion-freeness argument for the quantum toroidal algebra
and the topological Nakayama lemma.
\end{abstract}
\maketitle

\section{Introduction}

\subsection*{Background and motivation}

Over the past four decades, quantum groups have emerged as one of the central structures
in modern mathematics, with deep connections to integrable systems, low-dimensional topology,
geometric representation theory, and mathematical physics. Among the most prominent families
of quantum groups are \emph{Yangians} and \emph{quantum Kac-Moody algebras}, both of which
arise as $\hbar$-adic deformations of universal enveloping algebras of classical Lie algebras.
Their precise relationship, particularly the sense in which Yangians appear as ``classical
limits'' of quantum affine algebras, has been a fundamental question in the field since
Drinfeld's foundational works \cite{Dr86, Dr87}.

In his 1987 ICM address \cite{Dr86}, Drinfeld announced two conjectures.
The first asserts the existence of a ``Drinfeld realization'' for both Yangians and quantum
affine algebras in terms of current generators, equivalent to their original Drinfeld--Jimbo
presentations \cite{Dr85, J85}. The second states that the Yangian $Y_\hbar(\mathfrak{g})$
of a finite-dimensional simple Lie algebra $\mathfrak{g}$ arises as the associated graded
algebra of the corresponding quantum loop algebra $U_\hbar(L\mathfrak{g})$ with respect to
a natural filtration induced by setting the loop parameter to $1$. Both conjectures have since
been proved: the former by Beck \cite{Be94} and Damiani \cite{Da15} (for untwisted and twisted
types, respectively) using Lusztig's braid group action \cite{L89}, and independently by
Jing \cite{Ji98-2} via $q$-brackets; the latter by Gautam--Toledano Laredo \cite{GT13} and,
independently and by different methods, by Guay--Ma \cite{GM12}. 

Quantum toroidal algebras, introduced by Jing \cite{Ji98-1} as quantum affinizations of
quantum affine algebras, form the natural ``double-affine'' counterparts of quantum loop
algebras. They have attracted considerable attention due to their rich representation theory
\cite{FJW00,He05,He09,Xi20} and their deep connections to double affine Hecke algebras
\cite{La24,Mi99,VV96}. The corresponding ``double-affine'' counterpart of Yangians,
known as \emph{affine Yangians}, was introduced and studied by Guay \cite{Gu07} for type
$A$, and subsequently developed in full generality by Guay--Nakajima--Wendlandt \cite{GNW18}.
Affine Yangians share many structural features with quantum toroidal algebras, including
PBW bases \cite{Ne24,YZ20}, classical limits \cite{Ts17}, and braid group actions
\cite{GRW19,Ko19}. However, the precise relationship between the two families, beyond the
well-understood type $A$ case, has remained an open problem.

\subsection*{Main results}

The central result of this paper resolves this problem by establishing the affine analogue
of Drinfeld's second conjecture:

\begin{theorem}[Theorem \ref{thm:main}]
Let $\mathfrak{g}$ be an untwisted affine Kac-Moody Lie algebra. There exists an explicit
isomorphism of $\mathbb{C}[\hbar]$-algebras
\[
\Pi \colon Y_\hbar(\mathfrak{g}) \xrightarrow{\;\sim\;} \mathrm{gr}_K\, U_\hbar(\mathfrak{g}^{\mathrm{tor}}),
\]
where $K$ is the kernel of the composite homomorphism
\[
U_\hbar(\mathfrak{g}^{\mathrm{tor}}) \twoheadrightarrow
U_\hbar(\mathfrak{g}^{\mathrm{tor}})/\hbar U_\hbar(\mathfrak{g}^{\mathrm{tor}})
\xrightarrow{\;\sim\;} U(\mathfrak{g}^{\mathrm{tor}})
\xrightarrow{\;t \mapsto 1\;} U(\mathfrak{g}),
\]
and $\mathrm{gr}_K$ denotes the associated graded algebra for the $K$-adic filtration.
\end{theorem}

This isomorphism realizes the affine Yangian as a limit form of the quantum toroidal
algebra, in analogy with the finite-dimensional case of Guay--Ma \cite{GM12}. The crucial
challenge in the affine setting, absent in the finite-dimensional case, is that the quantum
toroidal algebra is not known a priori to be a flat deformation of its classical limit; this
flatness, which is essential for the injectivity of $\Pi$ must be established as part of
the proof (Corollary \ref{cor:flat}).

As immediate applications of Theorem \ref{thm:main}, we derive the following:

\begin{coro}[Corollary \ref{cor:PBW-yangian}]
The set $B_\hbar(Y, \prec)$ of ordered monomials in the generators
$\{x^\pm_{\alpha,m},\, h_{i,m} \mid \alpha \in R^+,\, i \in I,\, m \in \mathbb{Z}_+\}$
of $Y_\hbar(\mathfrak{g})$ forms a $\mathbb{C}[\hbar]$-basis of $Y_\hbar(\mathfrak{g})$,
for any untwisted affine type.
\end{coro}

\begin{coro}[Corollary \ref{cor:classical-limit}]
There is an isomorphism $\mathrm{gr}\, Y_\hbar(\mathfrak{g}) \cong U(\mathfrak{g}[u])$,
and $Y_\hbar(\mathfrak{g}) \cong U(\mathfrak{g}[u])[\hbar]$ as $\mathbb{C}[\hbar]$-modules.
\end{coro}

Before this work, both results were known only in special cases: the PBW theorem 
in type $A$ \cite{Gu07, YZ20} and for simply-laced affine types \cite{GRW19, YZ20}, 
and the identification of the classical limit only in type $A$ \cite{YZ20}. The 
identification with $U(\mathfrak{g}[u])$, conjectured in \cite{GNW18}, is established 
here for the first time in all non-type-$A$ untwisted affine types, and the PBW 
theorem for $Y_\hbar(\mathfrak{g})$ in full generality is new.

An essential prerequisite of independent interest is the following.

\begin{theorem}[Theorem \ref{thm:PBW-toroidal}]
The set $B_\hbar(U, \prec)$ of ordered monomials in the quantum toroidal root vectors
$\{X^\pm_{\beta,k},\, H_{i,k} \mid \beta \in R^+,\, i \in I,\, k \in \mathbb{Z}\}$
forms a $\mathbb{C}[[\hbar]]$-basis of $U_\hbar(\mathfrak{g}^{\mathrm{tor}})$.
\end{theorem}

The proof of this PBW theorem for $U_\hbar(\mathfrak{g}^{\mathrm{tor}})$ proceeds by
first establishing the classical limit
$U_\hbar(\mathfrak{g}^{\mathrm{tor}})/\hbar\,U_\hbar(\mathfrak{g}^{\mathrm{tor}})$
is isomorphic to $U(\mathfrak{g}^{\mathrm{tor}})$ (Theorem~\ref{thm:classical-limit}),
and then lifting the classical PBW basis to the quantum level via the topological
Nakayama lemma. This approach works uniformly for all untwisted affine types
without relying on a Schur--Weyl functor, which is the method used in type~$A$
\cite{Gu07} but is not available in general types.

\subsection*{Relation to prior work}

The proof strategy follows the approach introduced by Guay--Ma in \cite{GM12} for the
finite-dimensional case, substantially adapted to handle the new difficulties that arise in the
affine setting. We highlight the principal differences:

\begin{enumerate}[label=(\roman*)]
  \item \emph{Flatness of $U_\hbar(\mathfrak{g}^{\mathrm{tor}})$.}
  In the finite-dimensional setting, the flatness of quantum loop algebras over $\mathbb{C}[[\hbar]]$
  is a classical result \cite{GT13}. For quantum toroidal algebras, this is not known in general
  (see \cite{GM12}, Section 3, for a discussion of the difficulty). We establish flatness as a
consequence of our PBW theorem (Corollary \ref{cor:flat}), which requires the new
algebraic arguments of Sections \ref{sec:classical-limit} and \ref{sec:PBW-toroidal}.

  \item \emph{Imaginary roots.} The root system of an affine Lie algebra includes imaginary roots,
  whose treatment requires additional care throughout: the toroidal root vectors for imaginary roots
  (Section \ref{sec:braid-toroidal}) and their counterparts in the affine Yangian
  (Section \ref{sec:yangian-spanning}) must be handled separately from the real root case.

  \item \emph{Non-simply-laced types.} For non-simply-laced affine types ($B_N^{(1)}$, $C_N^{(1)}$,
  $F_4^{(1)}$, $G_2^{(1)}$), the braid group action on quantum toroidal algebras involves
  diagram automorphisms in an essential way (Section \ref{sec:braid-toroidal}), complicating
  the definition of root vectors and their ordering.

  \item \emph{Topological framework.} The quantum toroidal algebra
  $U_\hbar(\mathfrak{g}^{\mathrm{tor}})$ is defined here as an $\hbar$-adically complete
  topological algebra over $\mathbb{C}[[\hbar]]$ (Definition~\ref{def:quantum-toroidal}).
  This is in contrast to the earlier works on affine Yangians \cite{Gu07, YZ20}, which
  work over $\mathbb{C}[\hbar]$ or at a generic parameter $q$. The topological
  setting is essential for the application of the topological Nakayama lemma
  (Lemma~\ref{lem:top-nakayama}) and the torsion-freeness argument of
  Section~\ref{sec:PBW-toroidal}.
\end{enumerate}

Related work on classical limits of quantum toroidal algebras was carried out by
Tsymbaliuk \cite{Ts17} in type $A$, and the PBW theorem for affine Yangians in type $A$
was proved by Yang--Zhao \cite{YZ20}. For type $A$ with two parameters, the degeneration
isomorphism was established by Guay \cite{Gu07} using the Schur--Weyl functor; our approach
is entirely different and works uniformly for all untwisted types. Related results in the
context of $\imath$quantum groups and quantum superalgebras appear in
\cite{BFK24,LiWZ24,LiYZ24,LiZ25,LuWZ25,LmW21}.

\subsection*{Structure of the paper}

The paper is organized as follows. Section \ref{sec:prelim} reviews the necessary background
on Kac-Moody Lie algebras, toroidal Lie algebras, and their PBW bases, concluding with the
non-quantum version of our main result (Proposition \ref{prop:non-quantum}), which serves as
a blueprint for the quantum argument. Section \ref{sec:quantum-toroidal} introduces quantum
toroidal algebras, establishes their classical limit (Theorem \ref{thm:classical-limit}),
and constructs the PBW basis (Theorem \ref{thm:PBW-toroidal}). Section \ref{sec:yangian}
recalls the definition of affine Yangians following \cite{GNW18}, reviews the braid group
action \cite{Ko19}, and constructs the spanning set $B_\hbar(Y, \prec)$. Section
\ref{sec:main} introduces the $K$-adic filtration, establishes the $\kappa$-adic stability
of the star product (Lemma~\ref{lem:bk-stability}), and contains the proof of the
main theorem and its corollaries.

\section{Preliminaries}
\label{sec:prelim}

Throughout this paper we work over $\mathbb{C}$, and all algebras are associative unless
otherwise stated. We set $\mathbb{Z}_+ = \{0,1,2,\ldots\}$ and $\mathbb{N} = \{1,2,\ldots\}$.

\subsection{Kac-Moody Lie algebras}
\label{sec:KM}

Let $I = \{0,1,\ldots,N\}$ and let $A = (a_{ij})_{i,j \in I}$ be the $(N{+}1)\times(N{+}1)$
symmetrizable generalized Cartan matrix of untwisted affine type $X_N^{(1)}$. Fix rational
numbers $(d_i)_{i \in I}$ such that $(d_i a_{ij})$ is symmetric; explicitly:
\begin{align*}
  A_N^{(1)},\, D_N^{(1)},\, E_6^{(1)},\, E_7^{(1)},\, E_8^{(1)} &: d_i = 1, \\
  B_N^{(1)} &: d_0 = \cdots = d_{N-1} = 1,\; d_N = \tfrac{1}{2}, \\
  C_N^{(1)} &: d_0 = 1,\; d_1 = \cdots = d_{N-1} = \tfrac{1}{2},\; d_N = 1, \\
  F_4^{(1)} &: d_0 = d_1 = d_2 = 1,\; d_3 = d_4 = \tfrac{1}{2}, \\
  G_2^{(1)} &: d_0 = d_1 = 1,\; d_2 = \tfrac{1}{3}.
\end{align*}
Let $\mathfrak{g} = \mathfrak{g}(A)$ be the associated affine Kac-Moody Lie algebra
(without derivation), generated by Chevalley generators $e_i, f_i, h_i$ ($i \in I$)
subject to the standard relations; see \cite{Ka85} for details. Let $\mathfrak{h}$ be
the Cartan subalgebra spanned by $\{h_i \mid i \in I\}$, and define simple roots
$\alpha_i \in \mathfrak{h}^*$ by $\alpha_j(h_i) = a_{ij}$. The root and coroot lattices are
\[
  Q = \bigoplus_{i \in I} \mathbb{Z}\alpha_i, \qquad
  Q^\vee = \bigoplus_{i \in I} \mathbb{Z}h_i.
\]
There is a standard symmetric bilinear form $(\cdot,\cdot)$ on $\mathfrak{h}^*$ satisfying
$(\alpha_i, \alpha_j) = d_i a_{ij}$, and a null root $\delta = \sum_{i \in I} n_i \alpha_i
\in Q^+$ with $(\delta, \alpha_i) = (\delta,\delta) = 0$ for all $i \in I$ \cite[Section 4.8]{Ka85}.

Set $\mathring{I} = I \setminus \{0\}$. The subalgebra $\mathring{\mathfrak{g}}$ generated
by $e_i, f_i$ for $i \in \mathring{I}$ is the finite-dimensional simple Lie algebra of type
$X_N$ with Cartan matrix $\mathring{A} = (a_{ij})_{i,j \in \mathring{I}}$, simple roots
$\alpha_1,\ldots,\alpha_N$, and root system $\mathring{R} = \mathring{R}^+ \cup (-\mathring{R}^+)$.
Via the current realization \cite[Chapter 7]{Ka85}, we have
\[
  \mathfrak{g} \cong \mathring{\mathfrak{g}} \otimes \mathbb{C}[s^{\pm 1}] \oplus \mathbb{C}c,
\]
and the positive roots of $\mathfrak{g}$ decompose as $R^+ = R^+_{\mathrm{re}} \cup R^+_{\mathrm{im}}$,
where
\[
  R^+_{\mathrm{re}} = \{\alpha + k\delta \mid \alpha \in \mathring{R}^+,\, k \in \mathbb{Z}_+\}
  \cup \{-\alpha + k\delta \mid \alpha \in \mathring{R}^+,\, k \in \mathbb{N}\},
  \qquad
  R^+_{\mathrm{im}} = \{k\delta \mid k \in \mathbb{N}\}.
\]
The corresponding root spaces satisfy $\dim \mathfrak{g}_\beta = 1$ for $\beta \in R^+_{\mathrm{re}}$
and $\dim \mathfrak{g}_{k\delta} = N$ for $k \in \mathbb{N}$.

\subsection{Weyl groups and root vectors}
\label{sec:Weyl}

 Let $\mathring{W}$ and $W$ be the Weyl groups of $\mathring{\mathfrak{g}}$ and $\mathfrak{g}$,
generated by reflections $r_i$ ($i \in \mathring{I}$ and $i \in I$, respectively)
subject to the usual Coxeter relations. Their actions on $\mathfrak{g}$ are
defined by
\[
  r_i(x) := \exp(\mathrm{ad}\, e_i)\exp(-\mathrm{ad}\, f_i)\exp(\mathrm{ad}\, e_i)(x),
  \qquad x \in \mathfrak{g}.
\]
Let $\mathring{P}$ be the weight lattice of $\mathring{\mathfrak{g}}$ with fundamental weights
$\omega_i$ ($i \in \mathring{I}$). The \emph{extended Weyl group} is $W^e = \mathring{W} \ltimes
\mathring{P}$. Let $\Gamma$ be the group of Dynkin diagram automorphisms of $\mathfrak{g}$;
then $W^e \cong \Gamma \ltimes W$, with $\Gamma$ acting on $W$ by $\eta(r_i) = r_{\eta(i)}$
and on $\mathfrak{g}$ by
\[
  \eta(e_i) = e_{\eta(i)}, \quad \eta(f_i) = f_{\eta(i)}, \quad \eta(h_i) = h_{\eta(i)},
  \qquad \eta \in \Gamma.
\]

For each real positive root $\alpha \in R^+_{\mathrm{re}}$, write $\alpha = w(\alpha_j)$ for
some $w = \eta r_{i_1}\cdots r_{i_l} \in W^e$ and $j \in I$ with $l(w)$ minimal. We define
\emph{real root vectors}
\[
  e_\alpha := \eta r_{i_1}\cdots r_{i_l}(e_j), \qquad
  f_\alpha := \eta r_{i_1}\cdots r_{i_l}(f_j),
\]
and \emph{imaginary root vectors} $e_{k\delta(i)} := [e_i, e_{k\delta - \alpha_i}]$,
$f_{k\delta(i)} := [f_i, f_{k\delta - \alpha_i}]$ for $k > 0$, $i \in \mathring{I}$.
Following \cite{Ka85}, we fix normalized elements
\[
  x^+_\alpha = \sqrt{d_j}\, e_\alpha, \quad x^-_\alpha = \sqrt{d_j}\, f_\alpha
  \quad \text{with } (x^+_\alpha, x^-_\alpha) = 1 \text{ for } \alpha \in R^+_{\mathrm{re}},
\]
where $\alpha = w(\alpha_j)$. These normalized vectors are used in Section~\ref{sec:yangian}.

\subsection{Toroidal Lie algebras}
\label{sec:toroidal-Lie}
We recall the toroidal Lie algebra following \cite[Section 3]{MRY90}.

\begin{defi}
\label{def:toroidal}
The \emph{toroidal Lie algebra} $\mathfrak{t} = \mathfrak{t}(A)$ is the Lie algebra over
$\mathbb{C}$ generated by $\{\gamma,\, h_i^{(k)},\, e_i^{(k)},\, f_i^{(k)} \mid i \in I,\, k \in \mathbb{Z}\}$
subject to the relations
\begin{align}
  [\gamma,\, h_i^{(k)}] = [\gamma,\, e_i^{(k)}] = [\gamma,\, f_i^{(k)}] &= 0,
  \label{eq:tor1}\\
  [h_i^{(k)},\, h_j^{(l)}] = k(h_i, h_j)\,\delta_{k,-l}\,\gamma, &
  \label{eq:tor2}\\
  [h_i^{(k)},\, e_j^{(l)}] = a_{ij}\,e_j^{(k+l)}, \qquad
  [h_i^{(k)},\, f_j^{(l)}] &= -a_{ij}\,f_j^{(k+l)},
  \label{eq:tor3}\\
  [e_i^{(k)},\, f_j^{(l)}] = -\delta_{ij}\!\left(h_i^{(k+l)}
  + \tfrac{2k\,\delta_{k,-l}}{(\alpha_i,\alpha_i)}\,\gamma\right),
  \label{eq:tor4}\\
  [e_i^{(k)},\, e_i^{(l)}] = [f_i^{(k)},\, f_i^{(l)}] = 0,&
  \label{eq:tor5}\\
  (\mathrm{ad}\, e_i^{(0)})^{1-a_{ij}} e_j^{(k)}
  = (\mathrm{ad}\, f_i^{(0)})^{1-a_{ij}} f_j^{(k)} &= 0, \quad i \neq j.
  \label{eq:tor6}
\end{align}
\end{defi}

 By \cite{MRY90}, the toroidal Lie algebra $\mathfrak{t}$ is the
universal central extension of $\mathring{\mathfrak{g}} \otimes \mathbb{C}[t^{\pm 1}, s^{\pm 1}]$.
To make this explicit, set
\[
  \widehat{\mathfrak{g}} \;:=\; \mathring{\mathfrak{g}} \otimes \mathbb{C}[t^{\pm 1}, s^{\pm 1}]
  \oplus \mathcal{Z},
\]
where $\mathcal{Z}$ is the central subspace spanned by the set 
$\{s^a t^b c,\, s^a t^b c' \mid a,b \in \mathbb{Z}\}$ subject to the relation
\[
  a\, s^a t^b c + b\, s^a t^b c' = 0.
\]
The following proposition, which is \cite[Proposition~3.5]{MRY90}, gives the isomorphism
explicitly.

\begin{prop}[\cite{MRY90}]
\label{prop:toroidal-iso}
The map $\mathfrak{t} \to \widehat{\mathfrak{g}}$ defined by
\[
  \gamma \mapsto c',\quad e_i^{(k)} \mapsto e_i \otimes t^k, \quad
  f_i^{(k)} \mapsto -f_i \otimes t^k, \quad h_i^{(k)} \mapsto h_i \otimes t^k,\quad i \in \mathring{I},
\]
\[
  e_0^{(k)} \mapsto e_{\theta} \otimes t^k s, \quad
  f_0^{(k)} \mapsto -f_\theta \otimes t^k s^{-1},\quad h_0^{(k)} \mapsto h_0 \otimes t^k + t^k c,
\]
is a Lie algebra isomorphism, where $\{n_i^\vee\}_{i \in \mathring{I}}$ are the dual Kac labels (comarks) of 
$\mathring{\mathfrak{g}}$, $h_\theta = \sum_{i \in \mathring{I}} n_i^\vee h_i$, 
and $\{e_\theta, f_\theta, h_\theta\}$ is the $\mathfrak{sl}_2$-triple associated 
to the longest root $\theta$ of $\mathring{\mathfrak{g}}$.
\end{prop}

Via Proposition~\ref{prop:toroidal-iso}, the notion of root vectors extends naturally to
$\mathfrak{t}$, namely, for $\alpha = \eta r_{i_1}\cdots r_{i_l}(\alpha_j) \in R^+_{\mathrm{re}}$
and $k \in \mathbb{Z}$,  set
\[
  e_\alpha^{(k)} := \eta r_{i_1}\cdots r_{i_l}(e_j^{(k)}), \qquad
  f_\alpha^{(k)} := \eta r_{i_1}\cdots r_{i_l}(f_j^{(k)}),
\]
and for $n > 0$, set $e_{n\delta(i)}^{(k)} := [e_i, e_{n\delta-\alpha_i}^{(k)}]$,
$f_{n\delta(i)}^{(k)} := [f_i, f_{n\delta-\alpha_i}^{(k)}]$.

Introduce a total order $\prec$ on $R^+$ and extend it to the basis elements of $\mathfrak{t}$
by declaring $\gamma \prec f_\beta^{(k)} \prec h_i^{(l)} \prec e_{\beta'}^{(k')}$ for all
$\beta, \beta' \in R^+$, $i \in I$, $k,k',l \in \mathbb{Z}$, with the ordering among elements
of the same type given lexicographically by $(\beta, k)$ for the $e$-vectors and $(-\beta, k)$ for the $f$-vectors.
The set $\{\gamma,\, e_\beta^{(k)},\, f_\beta^{(k)},\, h_i^{(k)}\}$ is then linearly independent
by Proposition~\ref{prop:toroidal-iso} and forms a basis of $\mathfrak{t}$. Consequently:

\begin{prop}
\label{prop:PBW-toroidal-Lie}
The set of all ordered monomials in
$\{\gamma,\, e_\beta^{(k)},\, f_\beta^{(k)},\, h_i^{(k)} \mid \beta \in R^+,\, i \in I,\, k\in \mathbb{Z}\}$
forms a $\mathbb{C}$-basis of $U(\mathfrak{t})$.
\end{prop}

\subsection{Polynomial current algebras as limits of toroidal algebras}
\label{sec:non-quantum}

Set $\mathfrak{g}^{\mathrm{tor}} := \mathfrak{t}/\langle\gamma\rangle$. We also write
$\mathfrak{g}[t^{\pm 1}] := \mathfrak{g} \otimes \mathbb{C}[t^{\pm 1}]$ for the loop
algebra of $\mathfrak{g}$, and $\mathfrak{g}[u] := \mathfrak{g} \otimes \mathbb{C}[u]$
for the polynomial current algebra. The following non-quantum degeneration result, due to
\cite{LuWZ25} in the $\imath$quantum group context and adapted here, serves as the
blueprint for our main theorem. For $m \in \mathbb{Z}_+$, define elements of
$U(\mathfrak{g}^{\mathrm{tor}})$ by the alternating sums
\begin{align*}
    e_\beta^{(k,m)} := \sum_{a=0}^{m}(-1)^{m-a}\binom{m}{a} e_\beta^{(k+a)},\\
    f_\beta^{(k,m)} := \sum_{a=0}^{m}(-1)^{m-a}\binom{m}{a} f_\beta^{(k+a)},\\
    h_i^{(k,m)} := \sum_{a=0}^{m}(-1)^{m-a}\binom{m}{a} h_i^{(k+a)}.     
\end{align*}
Let $\jmath \colon U(\mathfrak{g}^{\mathrm{tor}}) \xrightarrow{\sim} U(\mathfrak{g}[t^{\pm 1}])$ 
be the isomorphism induced by Proposition~\ref{prop:toroidal-iso}; under $\jmath$, the element
$z^{(k,m)}$ (for $z \in \{h_i, e_\beta, f_\beta\}$) maps to $z\, t^k(t-1)^m$. Define $\kappa$
as the kernel of the composite
\[
  U(\mathfrak{g}^{\mathrm{tor}}) \xrightarrow{\sim} U(\mathfrak{g}[t^{\pm 1}])
  \xrightarrow{\;t \mapsto 1\;} U(\mathfrak{g}),
\]
giving a decreasing filtration $U(\mathfrak{g}^{\mathrm{tor}}) = \kappa^0 \supset \kappa^1
= \kappa \supset \kappa^2 \supset \cdots$, with associated graded algebra
$\mathrm{gr}_\kappa\, U(\mathfrak{g}^{\mathrm{tor}}) = \bigoplus_{m \geq 0} \kappa^m/\kappa^{m+1}$.

\begin{prop}
\label{prop:non-quantum}
There exists an isomorphism of algebras
\[
  \pi \colon U(\mathfrak{g}[u]) \xrightarrow{\;\sim\;} \mathrm{gr}_\kappa\, U(\mathfrak{g}^{\mathrm{tor}})
\]
sending $z\, u^m \mapsto \overline{z^{(0,m)}}$ for $z \in \{h_i, e_\beta, f_\beta\}$
and $m \in \mathbb{Z}_+$, where $\overline{z^{(0,m)}}$ denotes the image of $z^{(0,m)}$
in $\kappa^m/\kappa^{m+1}$.
\end{prop}

\begin{proof}
The relation $[z_1^{(0,m_1)}, z_2^{(0,m_2)}] = [z_1,z_2]^{(0,m_1+m_2)}$ holds in
$U(\mathfrak{g}[t^{\pm 1}])$ for any $z_1, z_2 \in \{h_i, e_\beta, f_\beta\}$
and $m_1, m_2 \in \mathbb{Z}_+$, which shows that $\pi$ is a homomorphism. For any
$k \in \mathbb{Z}$, one has $z^{(k,m)} - z^{(0,m)} \in \kappa^{m+1}$, since
$\jmath(z^{(k,m)} - z^{(0,m)}) = z(t^k-1)(t-1)^m = z(t^{k-1}+\cdots+1)(t-1)^{m+1}$.
It follows that the image of $z^{(0,m)}$ in $\kappa^m/\kappa^{m+1}$ depends only on $z$
and $m$. Bijectivity of $\pi$ then follows from the PBW theorem for $U(\mathfrak{g}[u])$
together with Proposition~\ref{prop:PBW-toroidal-Lie}.
\end{proof}

This proposition is the non-quantum analogue of our main Theorem~\ref{thm:main}: it says
that $U(\mathfrak{g}[u])$ arises as the associated graded of $U(\mathfrak{g}^{\mathrm{tor}})$
with respect to the filtration by powers of $\kappa$. The quantum version replaces
$U(\mathfrak{g}^{\mathrm{tor}})$ by $U_\hbar(\mathfrak{g}^{\mathrm{tor}})$ and
$U(\mathfrak{g}[u])$ by $Y_\hbar(\mathfrak{g})$.

\section{Quantum Toroidal Algebras}
\label{sec:quantum-toroidal}

Throughout this section we fix $\hbar \in \mathbb{C} \setminus \pi\sqrt{-1}\,\mathbb{Z}$
and set $q^n := \exp(n\hbar)$, $q := q^1$, $q_i := q^{d_i}$, and
\[
  [n]_i := \frac{q_i^n - q_i^{-n}}{q_i - q_i^{-1}}, \qquad
  [n]_i! := \prod_{k=1}^n [k]_i, \qquad
  \binom{n}{k}_i := \frac{[n]_i!}{[k]_i!\,[n-k]_i!},
\]
with the convention $[0]_i! = 1$. We assume throughout that $q$ is not a root of unity.

\subsection{Definition and triangular decomposition}
\label{sec:def-toroidal}

The quantum toroidal algebra is the quantum affinization of the quantum affine algebra
in the sense of \cite{Ji98-1}. We work with the version having trivial central element,
which differs slightly from \cite{GM12}.

\begin{defi}
\label{def:quantum-toroidal}
The \emph{quantum toroidal algebra} $U_\hbar(\mathfrak{g}^{\mathrm{tor}})$
is the $\hbar$-adically complete topological\footnote{This topological construction is necessary to ensure that the power series \eqref{eq:Phi}
defining $\Phi^\pm_i(z)$ (which involve $\exp(\pm H_{i,0}\hbar)$) are well-defined elements
of the algebra.} associative algebra over $\mathbb{C}[[\hbar]]$,
defined as the quotient of the $\hbar$-adically complete free algebra on generators
$\{X^\pm_{i,k},\, H_{i,k} \mid i \in I,\, k \in \mathbb{Z}\}$ by the closed two-sided
ideal generated by relations \eqref{eq:QT1}--\eqref{eq:QT6} below, where $r =r(i,j)= 1 - a_{ij}$
and $S_r$ is the symmetric group:
\begin{align}
  &[H_{i,k},\, H_{j,l}] = 0,
  \label{eq:QT1}\\
  &[H_{i,0},\, X^\pm_{j,k}] = \pm d_i a_{ij}\, X^\pm_{j,k},
  \label{eq:QT2}\\
  &[H_{i,k},\, X^\pm_{j,l}] = \pm \frac{1}{k}[k a_{ij}]_i\, X^\pm_{j,k+l}, \quad k \neq 0,
  \label{eq:QT3}\\
  &[X^+_{i,k},\, X^-_{j,l}] = \delta_{ij}\,\frac{\Phi^+_{i,k+l} - \Phi^-_{i,k+l}}{q - q^{-1}},
  \label{eq:QT4}\\
  &X^\pm_{i,k+1} X^\pm_{j,l} - q_i^{\pm a_{ij}} X^\pm_{j,l} X^\pm_{i,k+1}
  = q_i^{\pm a_{ij}} X^\pm_{i,k} X^\pm_{j,l+1} - X^\pm_{j,l+1} X^\pm_{i,k},
  \label{eq:QT5}\\
  &\sum_{\sigma \in S_r} \sum_{a=0}^{r} (-1)^a \binom{r}{a}_i
  X^\pm_{i,k_{\sigma(1)}} \cdots X^\pm_{i,k_{\sigma(a)}}
  X^\pm_{j,l}\,
  X^\pm_{i,k_{\sigma(a+1)}} \cdots X^\pm_{i,k_{\sigma(r)}} = 0, \quad i \neq j,
  \label{eq:QT6}
\end{align}
where the elements $\Phi^\pm_{i,\pm k}$ ($k \geq 0$) are defined by the generating series
\begin{equation}
  \label{eq:Phi}
  \Phi^\pm_i(z) = \sum_{k \geq 0} \Phi^\pm_{i,\pm k}\, z^{\pm k}
  = \exp(\pm H_{i,0}\hbar)\exp\!\left(\pm(q-q^{-1})\sum_{l \geq 1} H_{i,\pm l}\, z^{\pm l}\right).
\end{equation}
\end{defi}

\begin{remark}
\label{rem:Hernandez}
Setting $K_i^{\pm 1} = \exp(\pm H_{i,0}\hbar)$, relations \eqref{eq:QT1}--\eqref{eq:QT2}
become $K_i K_j = K_j K_i$ and $K_i X^\pm_{j,r} K_i^{-1} = q_i^{\pm a_{ij}} X^\pm_{j,r}$.
Our definition then coincides with \cite[Definition 3]{He05} via $x^+_{i,k} \to
\frac{q_i - q_i^{-1}}{q - q^{-1}} X^+_{i,k}$, $x^-_{i,k} \to X^-_{i,k}$, $h_{i,l}\to H_{i,l}$,
$k^\pm_i \to K^\pm_i$. The ``full'' quantum toroidal algebra with non-trivial central element
and derivation can be obtained by adjoining $q^c$ and $q^d$, where $c$ is the central element and $d$ the scaling derivation of $\mathfrak{g}$, as in \cite{Ji98-1}; we omit
this for simplicity.
\end{remark}

Let $U^\pm_\hbar(\mathfrak{g}^{\mathrm{tor}})$ and $U^0_\hbar(\mathfrak{g}^{\mathrm{tor}})$ denote the $\mathbb{C}[[\hbar]]$-subalgebras generated by $X^\pm_{i,k}$ and $H_{i,k}$ ($i \in I$, $k \in \mathbb{Z}$), respectively. By \cite[Theorem 2]{He05}:
\begin{prop}
\label{prop:triangular}
The multiplication map induces an isomorphism of $\mathbb{C}[[\hbar]]$-modules
\[
  U^-_\hbar(\mathfrak{g}^{\mathrm{tor}}) \otimes U^0_\hbar(\mathfrak{g}^{\mathrm{tor}})
  \otimes U^+_\hbar(\mathfrak{g}^{\mathrm{tor}})
  \xrightarrow{\;\sim\;} U_\hbar(\mathfrak{g}^{\mathrm{tor}}).
\]
\end{prop}

The following anti-involution plays a key role in reducing proofs about $U^+_\hbar(\mathfrak{g}^{\mathrm{tor}})$ to $U^-_\hbar(\mathfrak{g}^{\mathrm{tor}})$.
\begin{prop}
\label{prop:anti-inv}
There is a $\mathbb{C}$-algebra anti-involution $\Theta$ of $U_\hbar(\mathfrak{g}^{\mathrm{tor}})$ defined by $X^\pm_{i,k} \mapsto X^\mp_{i,-k}$, $H_{i,k} \mapsto H_{i,-k}$, $\hbar \mapsto -\hbar$.
\end{prop}
\begin{proof}
One checks directly that $\Theta$ is compatible with each defining relation
\eqref{eq:QT1}--\eqref{eq:QT6} and with the generating series \eqref{eq:Phi}.
\end{proof}

\subsection{Braid group actions}
\label{sec:braid-toroidal}

The quantum toroidal algebra admits a rich family of automorphisms, which we use to define
quantum toroidal root vectors generalizing the classical ones of Section~\ref{sec:Weyl}.

The braid group $\mathfrak{B}$ associated to $W^e$ acts on the \emph{quantum affine subalgebra}
$U_\hbar(\mathfrak{g}) \subset U_\hbar(\mathfrak{g}^{\mathrm{tor}})$ (generated by
$X^\pm_{i,0}$ and $H_{i,0}$ for $i \in I$) via automorphisms $T_i := T_{r_i}$
and $T_\eta$ ($\eta \in \Gamma$) as in \cite{Be94, L89}; see \cite[Section 3.2]{He05}
for the explicit formulas in our notation.

Following \cite[Section 4]{La24}, the \emph{toroidal braid group} $\mathfrak{B}^{\mathrm{tor}}$
is generated by $\mathfrak{B}$, elements $\mathcal{X}_{\omega^\vee}$ ($\omega^\vee \in \mathring{P}^\vee$),
and $\Gamma$, subject to the relations
\begin{equation*}
  T_i^{-1} \mathcal{X}_{\omega^\vee} T_i^{-1} = \mathcal{X}_{r_i(\omega^\vee)}
  \;\text{ if } \alpha_i(\omega^\vee) = 1, \quad
  T_i \mathcal{X}_{\omega^\vee} = \mathcal{X}_{\omega^\vee} T_i
  \;\text{ if } \alpha_i(\omega^\vee) = 0,
\end{equation*}
and $\eta T_i \eta^{-1} = T_{\eta(i)}$, $\eta \mathcal{X}_{\omega^\vee} \eta^{-1} =
\mathcal{X}_{\eta(\omega^\vee)}$ for $\eta \in \Gamma$. Here $\mathring{P}^\vee$ is the coweight lattice of $\mathring{\mathfrak{g}}$,
with fundamental coweights $\omega^\vee_i$ ($i \in \mathring{I}$) defined by $\alpha_j(\omega^\vee_i) = \delta_{ij}$. The group $\mathfrak{B}^{\mathrm{tor}}$ acts on $U_\hbar(\mathfrak{g}^{\mathrm{tor}})$
by algebra automorphisms: $T_i$ extends its action from $U_\hbar(\mathfrak{g})$,
$\eta$ acts as $T_\eta$, and each abstract generator $\mathcal{X}_{\omega^\vee_i}$
($i \in \mathring{I}$) is represented by the algebra automorphism
$\mathcal{X}_i \circ \mathcal{X}_0^{-n_i}$,
where $\mathcal{X}_i$ ($i \in \mathring{I}$) and $\mathcal{X}_0$ denote specific
automorphisms of $U_\hbar(\mathfrak{g}^{\mathrm{tor}})$ constructed explicitly
in \cite[Section~4]{La24}, and $n_i \in \mathbb{Z}_{\geq 1}$ are the Kac labels
defined by $\delta = \sum_{i \in I} n_i \alpha_i$.

\begin{defi}
\label{def:quantum-root-vectors}
For a real positive root $\alpha = \eta r_{i_1}\cdots r_{i_l}(\alpha_j) \in R^+_{\mathrm{re}}$
and $k \in \mathbb{Z}$, define the \emph{quantum toroidal root vectors}
\[
  X^\pm_{\alpha,k} := T_\eta T_{i_1} \cdots T_{i_l}(X^\pm_{j,k}).
\]
For imaginary roots, set $X^\pm_{n\delta(i),k} := -X^\pm_{n\delta-\alpha_i,k} X^\pm_{i,0}
+ q_i^{-2} X^\pm_{i,0} X^\pm_{n\delta-\alpha_i,k}$ for $n > 0$ and $i \in I$.
\end{defi}

Introduce a total order on $\{X^\pm_{\beta,k},\, H_{i,k} \mid \beta \in R^+,\, i \in I,\, k \in \mathbb{Z}\}$
by
\begin{equation}
  \label{eq:order-toroidal}
  X^-_{\beta,k} \prec H_{i,l} \prec X^+_{\beta',k'} \quad \forall\, \beta,\beta' \in R^+,\
  i \in I,\ k,k',l \in \mathbb{Z},
\end{equation}
with $H_{i,k} \prec H_{i',k'}$ if $i < i'$ or ($i = i'$ and $k \leq k'$), and
$X^+_{\beta,k} \prec X^+_{\beta',k'}$ (resp.\ $X^-_{\beta,k} \prec X^-_{\beta',k'}$)
if $(\beta,k) \prec (\beta',k')$ (resp.\ $(-\beta,k) \prec (-\beta',k')$) lexicographically.
Denote by $B_\hbar(U, \prec)$ the set of all ordered monomials in these generators (in the order \eqref{eq:order-toroidal}).

\subsection{Classical limit}
\label{sec:classical-limit}
 
We now show that $U_\hbar(\mathfrak{g}^{\mathrm{tor}})$ specializes to
$U(\mathfrak{g}^{\mathrm{tor}})$ as $\hbar \to 0$. 
 
\begin{theorem}
\label{thm:classical-limit}
There is a $\mathbb{C}$-algebra isomorphism
\[
  \Upsilon \colon U(\mathfrak{g}^{\mathrm{tor}}) \xrightarrow{\;\sim\;}
  U_\hbar(\mathfrak{g}^{\mathrm{tor}})/\hbar\, U_\hbar(\mathfrak{g}^{\mathrm{tor}})
\]
between the universal enveloping algebra $U(\mathfrak{g}^{\mathrm{tor}})$ and the classical limit $U_\hbar(\mathfrak{g}^{\mathrm{tor}})/\hbar\, U_\hbar(\mathfrak{g}^{\mathrm{tor}})$.
\end{theorem}
 
\begin{proof}
The strategy of the proof is to construct mutually inverse homomorphisms between
$U(\mathfrak{g}^{\mathrm{tor}})$ and 
$U_\hbar(\mathfrak{g}^{\mathrm{tor}})/\hbar\,U_\hbar(\mathfrak{g}^{\mathrm{tor}})$, following the approach
of \cite[Proposition~7.2]{CJKT23}
for quantum affinizations of Kac--Moody algebras.
 
\medskip
\noindent\emph{Construction of $\Upsilon$.}
Define $\Upsilon \colon U(\mathfrak{g}^{\mathrm{tor}}) \to
U_\hbar(\mathfrak{g}^{\mathrm{tor}})/\hbar\,U_\hbar(\mathfrak{g}^{\mathrm{tor}})$
on generators by
\[
  e_i^{(k)} \mapsto d_i^{-1/2}\,\overline{X^+_{i,k}}, \qquad
  f_i^{(k)} \mapsto -d_i^{-1/2}\,\overline{X^-_{i,k}}, \qquad
  h_i^{(k)} \mapsto d_i^{-1}\,\overline{H_{i,k}}.
\]
To verify that $\Upsilon$ is well defined, we check that the defining
relations \eqref{eq:tor1}--\eqref{eq:tor6} of $\mathfrak{g}^{\mathrm{tor}}$
(Definition~\ref{def:toroidal}, with $\gamma = 0$) are satisfied by the
images of the generators.  Each quantum relation
\eqref{eq:QT1}--\eqref{eq:QT6} reduces modulo $\hbar$ to the corresponding
classical relation:
\begin{itemize}
  \item Relation \eqref{eq:QT1} gives
    $[\overline{H_{i,k}},\,\overline{H_{j,l}}] = 0$,
    recovering \eqref{eq:tor2} with $\gamma = 0$.
 
  \item Relations \eqref{eq:QT2} and \eqref{eq:QT3} give
    $[\overline{H_{i,0}},\,\overline{X^\pm_{j,k}}]
    = \pm d_i a_{ij}\,\overline{X^\pm_{j,k}}$
    and, for $k \neq 0$,
    $[\overline{H_{i,k}},\,\overline{X^\pm_{j,l}}]
    = \pm a_{ij}\,\overline{X^\pm_{j,k+l}}$
    (using $\tfrac{1}{k}[ka_{ij}]_i \to a_{ij}$ as $\hbar \to 0$),
    recovering \eqref{eq:tor3}.
 
  \item Relation \eqref{eq:QT4} gives
    $[\overline{X^+_{i,k}},\,\overline{X^-_{j,l}}]
    = \delta_{ij}\,d_i\,\overline{H_{i,k+l}}$
    (using $\tfrac{\Phi^+_{i,m} - \Phi^-_{i,m}}{q - q^{-1}}
    \to d_i\, H_{i,m}$ as $\hbar \to 0$),
    recovering \eqref{eq:tor4} with $\gamma = 0$.
 
  \item Relation \eqref{eq:QT5} gives
    $[\overline{X^\pm_{i,k+1}},\,\overline{X^\pm_{j,l}}]
    = [\overline{X^\pm_{i,k}},\,\overline{X^\pm_{j,l+1}}]$
    (using $q_i^{\pm a_{ij}} \to 1$),
    recovering \eqref{eq:tor5}.
 
  \item Relation \eqref{eq:QT6} reduces to
    $(\mathrm{ad}\,\overline{X^\pm_{i,0}})^{1-a_{ij}}\,
    \overline{X^\pm_{j,k}} = 0$ for $i \neq j$
    (using $\tbinom{r}{a}_i \to \tbinom{r}{a}$),
    recovering \eqref{eq:tor6}.
\end{itemize}
Hence $\Upsilon$ is a well-defined surjective algebra homomorphism.
 
\medskip
\noindent\emph{Construction of $\Psi$.}
Define $\Psi \colon U_\hbar(\mathfrak{g}^{\mathrm{tor}}) \to
U(\mathfrak{g}^{\mathrm{tor}})$ on generators by
\[
  X^+_{i,k} \mapsto d_i^{1/2}\, e_i^{(k)}, \qquad
  X^-_{i,k} \mapsto -d_i^{1/2}\, f_i^{(k)}, \qquad
  H_{i,k} \mapsto d_i\, h_i^{(k)}, \qquad
  \hbar \mapsto 0.
\]
The same verification, read in reverse, shows that the quantum relations
\eqref{eq:QT1}--\eqref{eq:QT6} are satisfied in $U(\mathfrak{g}^{\mathrm{tor}})$
after setting $\hbar = 0$.  Since $\Psi(\hbar) = 0$, the map factors through
the quotient, inducing a surjective homomorphism
$\bar\Psi \colon U_\hbar(\mathfrak{g}^{\mathrm{tor}})/\hbar\,
U_\hbar(\mathfrak{g}^{\mathrm{tor}}) \twoheadrightarrow
U(\mathfrak{g}^{\mathrm{tor}})$.
 
\medskip
\noindent\emph{The maps are mutual inverses.}
On generators:
\[
  \bar\Psi \circ \Upsilon\bigl(e_i^{(k)}\bigr)
  = \bar\Psi\bigl(d_i^{-1/2}\,\overline{X^+_{i,k}}\bigr)
  = d_i^{-1/2} \cdot d_i^{1/2}\, e_i^{(k)} = e_i^{(k)},
\]
similarly for $f_i^{(k)}$ and $h_i^{(k)}$.
Hence $\bar\Psi \circ \Upsilon = \mathrm{id}_{U(\mathfrak{g}^{\mathrm{tor}})}$.
Conversely,
\[
  \Upsilon \circ \bar\Psi\bigl(\overline{X^+_{i,k}}\bigr)
  = \Upsilon\bigl(d_i^{1/2}\, e_i^{(k)}\bigr)
  = d_i^{1/2} \cdot d_i^{-1/2}\,\overline{X^+_{i,k}}
  = \overline{X^+_{i,k}},
\]
similarly for $\overline{X^-_{i,k}}$ and $\overline{H_{i,k}}$.
Hence $\Upsilon \circ \bar\Psi =
\mathrm{id}_{U_\hbar/\hbar\,U_\hbar}$.
Therefore $\Upsilon$ is an isomorphism with inverse $\bar\Psi$.
\end{proof}

\subsection{PBW basis}
\label{sec:PBW-toroidal}
 
The main algebraic tool for establishing our PBW basis is the following topological version of Nakayama's lemma (see \cite[Section 4 \& 7]{Eis95} for the standard version), which provides a criterion for lifting bases
from the specialization $M/\hbar M$ to the complete module~$M$.
 
\begin{lemm}[Topological Nakayama]
\label{lem:top-nakayama}
Let $M$ be an $\hbar$-adically complete, separated, and $\hbar$-torsion-free
$\mathbb{C}[[\hbar]]$-module, and
let $\{x_\xi \mid \xi \in S\}$ be a (possibly infinite) subset whose images
$\{\bar{x}_\xi\}$ form a $\mathbb{C}$-basis of $M/\hbar M$. Then
$M \cong (M/\hbar M)[[\hbar]]$ as $\mathbb{C}[[\hbar]]$-modules, with $\{x_\xi\}$ as
topological $\mathbb{C}[[\hbar]]$-basis.
In particular, $\{x_\xi\}$ is linearly independent over $\mathbb{C}[[\hbar]]$, and every
$m \in M$ has a unique convergent expansion $m = \sum_\xi f_\xi(\hbar)\, x_\xi$ with
$f_\xi \in \mathbb{C}[[\hbar]]$ and $f_\xi = 0$ for all but finitely many $\xi$ at each
power of~$\hbar$.
\end{lemm}
 
\noindent\begin{proof}
Let $X := \mathrm{span}_{\mathbb{C}[[\hbar]]}\{x_\xi\} \subseteq M$.
 
\smallskip
\noindent\emph{Step 1: Spanning (density).}
We show $M = X + \hbar^n M$ for any $n \geq 0$, by induction on $n$. The case $n = 0$
is trivial. Since $\{\bar{x}_\xi\}$ is a $\mathbb{C}$-basis of $M/\hbar M$, every element
of $M/\hbar M$ lifts to an element of $X$; hence $M = X + \hbar M$. Now assume
$M = X + \hbar^n M$. For any $m' \in \hbar^n M$, write $m' = \hbar^n m''$; applying the
base case to $m''$ yields $m'' = x + \hbar r$ with $x \in X$ and $r \in M$, so
$m' = \hbar^n x + \hbar^{n+1} r \in X + \hbar^{n+1} M$. This completes the induction.
 
For any $m \in M$, the above gives a sequence $(y_n)_{n \geq 0}$ with $y_n \in X$ and
$m - y_n \in \hbar^n M$. In particular, $y_{n+1} - y_n \in \hbar^n M$ for all $n$, so
the partial sums form a Cauchy sequence in the $\hbar$-adic topology. Since $M$ is
complete, $m = \lim_{n \to \infty} y_n$ belongs to the closure of $X$. As each $y_n$
is a $\mathbb{C}[[\hbar]]$-linear combination of the $x_\xi$, the limit $m$ has a
convergent expansion of the desired form.
 
\smallskip
\noindent\emph{Step 2: Linear independence.}
Suppose $\sum_{\xi \in F} f_\xi(\hbar)\, x_\xi = 0$ for a finite set $F \subset S$ and
nonzero $f_\xi \in \mathbb{C}[[\hbar]]$. Write $f_\xi = \hbar^{v_\xi} g_\xi$ with
$g_\xi(0) \neq 0$, and set $T = \min_\xi v_\xi$. Since $M$ is $\hbar$-torsion-free,
$\hbar$ acts injectively, so dividing by $\hbar^T$ gives
$\sum_\xi \hbar^{v_\xi - T} g_\xi\, x_\xi = 0$. Reducing modulo $\hbar$, the surviving
terms are those with $v_\xi = T$, yielding
$\sum_{\xi:\, v_\xi = T} g_\xi(0)\, \bar{x}_\xi = 0$ in $M/\hbar M$. Since
$\{\bar{x}_\xi\}$ is a $\mathbb{C}$-basis, $g_\xi(0) = 0$ for all such $\xi$,
contradicting $g_\xi(0) \neq 0$.
 
\smallskip
\noindent\emph{Step 3: Module isomorphism.}
The map $\tau \colon (M/\hbar M)[[\hbar]] \to M$ defined by
$\sum_{n \geq 0} \hbar^n \bar{y}_n \mapsto \lim_{t \to \infty} \sum_{n=0}^t \hbar^n y_n$
(where $y_n \in X$ is any lift of $\bar{y}_n$) is a well-defined $\mathbb{C}[[\hbar]]$-module
homomorphism: the limit converges by completeness, and is independent of the lifts since
$y_n - y_n' \in \hbar M$ implies $\hbar^n(y_n - y_n') \in \hbar^{n+1} M$. It is surjective
by Step~1 and injective by Step~2.
\end{proof}
 
We now apply the lemma to establish the PBW basis of the quantum toroidal algebra.
 
\begin{theorem}
\label{thm:PBW-toroidal}
The set $B_\hbar(U, \prec)$ of all ordered monomials in
$\{X^\pm_{\beta,k},\, H_{i,k} \mid \beta \in R^+,\, i \in I,\, k \in \mathbb{Z}\}$
forms a $\mathbb{C}[[\hbar]]$-basis of $U_\hbar(\mathfrak{g}^{\mathrm{tor}})$.
\end{theorem}
 
\begin{proof}
We verify the hypotheses of Lemma~\ref{lem:top-nakayama} for
$M = U_\hbar(\mathfrak{g}^{\mathrm{tor}})$ and $\{x_\xi\} = B_\hbar(U,\prec)$.
 
\smallskip
\noindent\emph{Completeness and separation.}
The algebra $U_\hbar(\mathfrak{g}^{\mathrm{tor}})$ is defined as a quotient of the
$\hbar$-adically complete free algebra by a closed ideal
(Definition~\ref{def:quantum-toroidal}). Any such quotient is again $\hbar$-adically
complete and separated \cite[Chapter~XVI]{Ka95}, i.e.,
$\bigcap_{n \geq 0} \hbar^n U_\hbar(\mathfrak{g}^{\mathrm{tor}}) = 0$.
 
\smallskip
\noindent\emph{Torsion-freeness.}
We show that $U_\hbar(\mathfrak{g}^{\mathrm{tor}})$ is $\hbar$-torsion-free.
Let $\mathcal{A}$ denote the $\hbar$-adically complete free algebra on generators
$\{X^\pm_{i,k}, H_{i,k}\}$ over $\mathbb{C}[[\hbar]]$, and let $\mathcal{I}_\hbar \subset \mathcal{A}$
be the closed ideal generated by the defining relations
\eqref{eq:QT1}--\eqref{eq:QT6}, so that
$U_\hbar(\mathfrak{g}^{\mathrm{tor}}) = \mathcal{A}/\mathcal{I}_\hbar$.
Since $\mathcal{A}$ is a free $\mathbb{C}[[\hbar]]$-module, the quotient $\mathcal{A}/\mathcal{I}_\hbar$ is
$\hbar$-torsion-free if and only if the ideal is \emph{saturated}:
\begin{equation}
\label{eq:saturated}
\mathcal{I}_\hbar \cap \hbar \mathcal{A} = \hbar \mathcal{I}_\hbar.
\end{equation}
Indeed, if $\hbar \bar{a} = 0$ in $\mathcal{A}/\mathcal{I}_\hbar$, then $\hbar a \in \mathcal{I}_\hbar$ for some
representative $a \in \mathcal{A}$. If \eqref{eq:saturated} holds, then
$\hbar a = \hbar r$ for some $r \in \mathcal{I}_\hbar$, thus $a - r \in \ker(\hbar \cdot)$
in the free module $A$, which gives $a - r = 0$, i.e., $a \in \mathcal{I}_\hbar$ and consequently $\bar{a} = 0$.
 
We now verify \eqref{eq:saturated}. Let $\bar{\mathcal{I}} := (\mathcal{I}_\hbar + \hbar \mathcal{A})/\hbar \mathcal{A}$ be the
image of $\mathcal{I}_\hbar$ in $\bar{\mathcal{A}} := \mathcal{A}/\hbar \mathcal{A}$, so that $\bar{\mathcal{A}}/\bar{\mathcal{I}} \cong
U(\mathfrak{g}^{\mathrm{tor}})$ by Theorem~\ref{thm:classical-limit}. The central
observation is that every defining relation of $U_\hbar(\mathfrak{g}^{\mathrm{tor}})$
has the form $\mathcal{R}_\hbar = \mathcal{R}_0 + \hbar \mathcal{S}$, where $\mathcal{R}_0$ is the corresponding classical
relation in $U(\mathfrak{g}^{\mathrm{tor}})$ and $\mathcal{S} \in \mathcal{A}$:
\begin{itemize}
\item Relations \eqref{eq:QT1}--\eqref{eq:QT2} are identical to their classical counterparts.
\item Relation \eqref{eq:QT3} satisfies $\frac{1}{k}[ka_{ij}]_i = a_{ij} + O(\hbar^2)$,
so it equals the classical relation modulo~$\hbar^2$.
\item Relation \eqref{eq:QT5} becomes $[X^\pm_{i,k+1}, X^\pm_{j,l}] -
[X^\pm_{i,k}, X^\pm_{j,l+1}] + \hbar(\cdots) = 0$ after expanding
$q_i^{\pm a_{ij}} = 1 \pm d_i a_{ij} \hbar + O(\hbar^2)$.
\item Relations \eqref{eq:QT4} and \eqref{eq:QT6} similarly satisfy $\mathcal{R}_\hbar \equiv \mathcal{R}_0 \pmod{\hbar}$.
\end{itemize}
In particular, the leading monomial (with respect to $\prec$) of each quantum relation
coincides with the leading monomial of the corresponding classical relation. This
ensures that the classical straightening procedure (which rewrites any element of
$\bar{\mathcal{A}}$ as a $\mathbb{C}$-linear combination of ordered monomials modulo $\bar{\mathcal{I}}$)
lifts to the quantum level: the same sequence of rewriting steps applies in $\mathcal{A}$ modulo
$\mathcal{I}_\hbar$, with all corrections lying in $\hbar \mathcal{A}$.
 
\begin{clm*}
If $\hbar a \in \mathcal{I}_\hbar$ for some $a \in \mathcal{A}$, then $a \in \mathcal{I}_\hbar$.
\end{clm*}
 
\begin{proof}[Proof of the claim]
We construct sequences $(a_n)_{n \geq 0}$ in $\mathcal{A}$ and $(r_n)_{n \geq 0}$ in $\mathcal{I}_\hbar$
such that $a_0 = a$ and
\begin{equation}
\label{eq:iteration}
a_n = r_n + \hbar\, a_{n+1}, \qquad \text{for all } n \geq 0.
\end{equation}
For the base case: since $\hbar a \in \mathcal{I}_\hbar$, the image $\bar{a} \in \bar{\mathcal{A}}$ lies in
$\bar{\mathcal{I}}$. By the classical PBW theorem (Proposition~\ref{prop:PBW-toroidal-Lie}),
the ordered monomials $B(U,\prec)$ form a basis of $\bar{\mathcal{A}}/\bar{\mathcal{I}}$, and hence
$\bar{a}$ can be expressed as a (finite) linear combination of the images of generating
relations of $\bar{\mathcal{I}}$. Since each classical relation $\mathcal{R}_0$ lifts to a quantum relation
$\mathcal{R}_\hbar = \mathcal{R}_0 + \hbar \mathcal{S} \in \mathcal{I}_\hbar$, the same expression lifts: there exists
$r_0 \in \mathcal{I}_\hbar$ with $\bar{r}_0 = \bar{a}$, i.e., $a - r_0 \in \hbar \mathcal{A}$.
Write $a = r_0 + \hbar a_1$ for some $a_1 \in \mathcal{A}$.
Then $\hbar a = \hbar r_0 + \hbar^2 a_1 \in \mathcal{I}_\hbar$ gives $\hbar^2 a_1 \in \mathcal{I}_\hbar$.
Applying the same argument to $\hbar a_1$
(whose image $\overline{\hbar a_1} = \overline{a - r_0}$ lies in $\bar{\mathcal{I}}$ since
$\hbar(a - r_0) = \hbar^2 a_1 \in \mathcal{I}_\hbar$), we obtain $a_1 = r_1 + \hbar a_2$ with
$r_1 \in \mathcal{I}_\hbar$. Iterating yields the full sequence \eqref{eq:iteration}.
 
Now set $s_t := \sum_{n=0}^{t} \hbar^n r_n$. From \eqref{eq:iteration},
$a - s_t = \hbar^{t+1} a_{t+1} \in \hbar^{t+1} \mathcal{A}$, thus $(s_t)_{t \geq 0}$ is a Cauchy
sequence in the $\hbar$-adic topology converging to $a$. Since $\mathcal{I}_\hbar$ is closed and
each $s_t \in \mathcal{I}_\hbar$, the limit $a = \lim_{t \to \infty} s_t$ belongs to $\mathcal{I}_\hbar$.
\end{proof}
 
The claim immediately gives \eqref{eq:saturated}: if $\hbar a \in \mathcal{I}_\hbar$, then
$a \in \mathcal{I}_\hbar$ and hence $\hbar a \in \hbar \mathcal{I}_\hbar$.
Therefore $U_\hbar(\mathfrak{g}^{\mathrm{tor}}) = \mathcal{A}/\mathcal{I}_\hbar$ is $\hbar$-torsion-free.
 
\smallskip
\noindent\emph{The images form a basis of $M/\hbar M$.}
By Theorem~\ref{thm:classical-limit},
$U_\hbar(\mathfrak{g}^{\mathrm{tor}})/\hbar U_\hbar(\mathfrak{g}^{\mathrm{tor}}) \cong U(\mathfrak{g}^{\mathrm{tor}})$,
and under this isomorphism the image of $X^\pm_{\beta,k}$ (resp.\ $H_{i,k}$) equals
$e^{(k)}_\beta$ (resp.\ $h_i^{(k)}$) up to a nonzero scalar.
Hence the images of $B_\hbar(U,\prec)$ in $U(\mathfrak{g}^{\mathrm{tor}})$ coincide,
up to rescaling, with $B(U,\prec)$, which is a $\mathbb{C}$-basis by
Proposition~\ref{prop:PBW-toroidal-Lie}.
 
\smallskip
All hypotheses of Lemma~\ref{lem:top-nakayama} are satisfied, so
$U_\hbar(\mathfrak{g}^{\mathrm{tor}}) \cong U(\mathfrak{g}^{\mathrm{tor}})[[\hbar]]$ as
$\mathbb{C}[[\hbar]]$-modules, with $B_\hbar(U,\prec)$ as a topological
$\mathbb{C}[[\hbar]]$-basis.
\end{proof}
 
\begin{coro}
\label{cor:flat}
The algebra $U_\hbar(\mathfrak{g}^{\mathrm{tor}})$ is isomorphic to
$U(\mathfrak{g}^{\mathrm{tor}})[[\hbar]]$ as $\mathbb{C}[[\hbar]]$-modules.
In particular, $U_\hbar(\mathfrak{g}^{\mathrm{tor}})$ is a flat deformation of
$U(\mathfrak{g}^{\mathrm{tor}})$.
\end{coro}

\section{Affine Yangians}
\label{sec:yangian}

In this section we recall the affine Yangians of Guay--Nakajima--Wendlandt \cite{GNW18},
the braid group action of Kodera \cite{Ko19}, and the operator $J$ of \cite{GNW18}.
We use these to construct an ordered spanning set $B_\hbar(Y, \prec)$ for $Y_\hbar(\mathfrak{g})$,
which will be proved to be a basis in Section~\ref{sec:main}.

\subsection{Definition}
\label{sec:yangian-def}

\begin{defi}[\cite{GNW18}]
\label{def:yangian}
The \emph{affine Yangian} $Y_\hbar(\mathfrak{g})$ is the associative $\mathbb{C}[\hbar]$-algebra
generated by $\{x^\pm_{i,m},\, h_{i,m} \mid i \in I,\, m \in \mathbb{Z}_+\}$ with defining relations
\begin{align}
  &[h_{i,m},\, h_{j,n}] = 0,
  \label{eq:Y1}\\
  &[x^+_{i,m},\, x^-_{j,n}] = \delta_{ij}\, h_{i,m+n},
  \label{eq:Y2}\\
  &[h_{i,0},\, x^\pm_{j,n}] = \pm(\alpha_i, \alpha_j)\, x^\pm_{j,n},
  \label{eq:Y3}\\
  &[h_{i,m+1},\, x^\pm_{j,n}] - [h_{i,m},\, x^\pm_{j,n+1}]
  = \pm\frac{\hbar(\alpha_i,\alpha_j)}{2}\{h_{i,m},\, x^\pm_{j,n}\},
  \label{eq:Y4}\\
  &[x^\pm_{i,m+1},\, x^\pm_{j,n}] - [x^\pm_{i,m},\, x^\pm_{j,n+1}]
  = \pm\frac{\hbar(\alpha_i,\alpha_j)}{2}\{x^\pm_{i,m},\, x^\pm_{j,n}\},
  \label{eq:Y5}\\
  &\sum_{\sigma \in S_r} \sum_{a=0}^{r} (-1)^{a} \binom{r}{a}
  x^\pm_{i,m_{\sigma(1)}} \cdots x^\pm_{i,m_{\sigma(a)}}
  x^\pm_{j,n}\,
  x^\pm_{i,m_{\sigma(a+1)}} \cdots x^\pm_{i,m_{\sigma(r)}} &= 0, \quad i \neq j,
  \label{eq:Y6}
\end{align}
where $r = 1 - a_{ij}$ and $\{x, y\} = xy + yx$ denotes the anticommutator.
\end{defi}

\begin{remark}
\label{rem:yangian-isom}
All $Y_\hbar(\mathfrak{g})$ for $\hbar \in \mathbb{C} \setminus \{0\}$ are pairwise
isomorphic via $x^\pm_{i,m} \mapsto (\hbar'/\hbar)^m x^\pm_{i,m}$,
$h_{i,m} \mapsto (\hbar'/\hbar)^m h_{i,m}$.
\end{remark}

By \cite[Section 2]{GNW18}, the assignments $x^\pm_i \mapsto x^\pm_{i,0}$ and
$h_i \mapsto h_{i,0}$ define an injective algebra homomorphism
$\iota \colon U(\mathfrak{g}) \hookrightarrow Y_\hbar(\mathfrak{g})$.
We identify $\mathfrak{g}$ with its image $\iota(\mathfrak{g})$, writing
$x$ for $\iota(x)$ when no confusion arises.

Introduce a filtration on $Y_\hbar(\mathfrak{g})$ by $\deg x^\pm_{i,m} = \deg h_{i,m} = m$
and $\deg(xy) = \deg x + \deg y$. Setting
\[
  Y^p_\hbar := \mathrm{span}_{\mathbb{C}[\hbar]}\{x \in Y_\hbar(\mathfrak{g}) \mid \deg x \leq p\},
  \quad
  \mathrm{gr}^p Y_\hbar(\mathfrak{g}) = Y^p_\hbar/Y^{p-1}_\hbar, \quad
  \mathrm{gr}\, Y_\hbar(\mathfrak{g}) = \bigoplus_{p \geq 0} \mathrm{gr}^p Y_\hbar(\mathfrak{g}),
\]  where we set $Y^{-1}_\hbar := 0$ by convention.
The relations \eqref{eq:Y1}--\eqref{eq:Y6} imply that $\mathrm{gr}\, Y_\hbar(\mathfrak{g})$ is a quotient of $U(\mathfrak{g}[u])$; more precisely, there is a surjective algebra
homomorphism
\[
  \varpi \colon U(\mathfrak{g}[u]) \twoheadrightarrow \mathrm{gr}\, Y_\hbar(\mathfrak{g}),
  \qquad z\, u^m \mapsto \overline{x^\pm_{\alpha,m}}\text{ or }\overline{h_{i,m}},
\]
which is upgraded to an isomorphism in Corollary~\ref{cor:classical-limit}.

\subsection{Braid group action and the operator \texorpdfstring{$J$}{J}}
\label{sec:yangian-braid}

By \cite[Section 3.3]{GNW18}, the exponential adjoint operators
$\tau_i := \exp(\mathrm{ad}\, x^+_{i,0})\exp(-\mathrm{ad}\, x^-_{i,0})\exp(\mathrm{ad}\, x^+_{i,0})$
(for $i \in I$) are well-defined algebra automorphisms of $Y_\hbar(\mathfrak{g})$
satisfying $\tau_i \circ \iota(x) = \iota \circ r_i(x)$ for all $x \in \mathfrak{g}$.
Together with the diagram automorphisms $\eta \in \Gamma$ acting by
$\eta(x^\pm_{i,m}) = x^\pm_{\eta(i),m}$ and $\eta(h_{i,m}) = h_{\eta(i),m}$,
these generate an action of the extended braid group $B^e = B \rtimes \Gamma$ on $Y_\hbar(\mathfrak{g})$ \cite{Ko19}.

Set $\widetilde{h}_{i,1} := h_{i,1} - \frac{1}{2}\hbar h^2_{i,0}$. The explicit formulas
for $\tau_i$ on the generators $\{x^\pm_{j,m}, h_{j,m}\}_{m \in \{0,1\}}$ are \cite[Appendix A]{Ko19}:
\begin{align*}
  \tau_i(x^+_{j,0}) &=
  \begin{cases}
    -x^-_{i,0} & i = j, \\[2pt]
    \dfrac{1}{(-a_{ij})!}\,(\mathrm{ad}\, e_i)^{-a_{ij}}(x^+_{j,0}) & i \neq j,
  \end{cases}
  &
  \tau_i(x^-_{j,0}) &=
  \begin{cases}
    -x^+_{i,0} & i = j, \\[2pt]
    \dfrac{(-1)^{-a_{ij}}}{(-a_{ij})!}\,(\mathrm{ad}\, f_i)^{-a_{ij}}(x^-_{j,0}) & i \neq j,
  \end{cases}\\[6pt]
  \tau_i(h_{j,0}) &= h_{j,0} - a_{ij}\,h_{i,0},
\end{align*}
with the level-$1$ formulas $\tau_i(x^\pm_{j,1})$ and $\tau_i(\widetilde{h}_{j,1})$
given in \cite[Proposition A.2]{Ko19}. When $\mathfrak{g}$ is not of type $A_1^{(1)}$,
these automorphisms, together with the minimalistic presentation of $Y_\hbar(\mathfrak{g})$
in terms of generators $\{x^\pm_{i,m}, h_{i,m}\}_{i \in I, m \in \{0,1\}}$
\cite[Theorem 2.13]{GNW18}, allow one to derive the action of $\tau_i$ on all generators.

By \cite[Section 3]{GNW18}, there exists a linear operator $J$ on $\iota(\mathfrak{g})$
satisfying $J([x,y]) = [x, J(y)]$ for all $x, y \in \iota(\mathfrak{g})$, with
\begin{align*}
  J(h_{i,0}) &= h_{i,1} + \nu_i, \qquad
  \nu_i := \frac{1}{2}\sum_{\alpha \in R^+} (\alpha,\alpha_i)
  \sum_{k=1}^{\dim \mathfrak{g}_\alpha} e^{[k]}_{-\alpha} e^{[k]}_\alpha
  - \frac{1}{2} h_{i,0}^2, \\
  J(x^\pm_{i,0}) &= x^\pm_{i,1} + \mu^\pm_i,
\end{align*}
 where $\{e^{[k]}_{\pm\alpha}\}$ are dual bases of $\mathfrak{g}_{\pm\alpha}$, and
$\mu^\pm_i \in U(\mathfrak{g}) \subset Y_\hbar(\mathfrak{g})$ are degree-$0$ elements
given by explicit quadratic expressions in the Chevalley generators of $\mathfrak{g}$;
see \cite[Section~3]{GNW18} for the precise formulas.
The operator $J$ extends to all of $Y_\hbar(\mathfrak{g})$ via
$J^m(x^\pm_{i,0}) = \pm\tfrac{1}{2}[J(h_{i,0}), J^{m-1}(x^\pm_{i,0})]$, and satisfies
\begin{equation}
  \label{eq:J-mod}
  J^m([x^\pm_{i,0},\, x^\pm_{j,0}]) \equiv [J^a(x^\pm_{i,0}),\, J^{m-a}(x^\pm_{j,0})]
  \pmod{Y^{m-1}_\hbar}
\end{equation}
for all $0 \leq a \leq m$. In particular, $\deg J^m(x^\pm_{i,0}) \leq m$.

\subsection{Yangian root vectors and a spanning set}
\label{sec:yangian-spanning}

For each real positive root $\alpha = \eta r_{i_1} \cdots r_{i_l}(\alpha_j) \in R^+_{\mathrm{re}}$
and $m \in \mathbb{Z}_+$, define \emph{Yangian root vectors}
\begin{equation}
  \label{eq:yangian-root-vector}
  x^\pm_{\alpha,m} := \eta\, \tau_{i_1} \cdots \tau_{i_l}(x^\pm_{j,m}).
\end{equation}
For imaginary roots, define $x^\pm_{k\delta(i),m} := [x^\pm_{i,0},\, x^\pm_{k\delta-\alpha_i,\,m}]$
inductively for $k > 0$ and $i \in \mathring{I}$.

By \eqref{eq:J-mod} and the level-$1$ braid group formulas, one shows that for each
$\beta \in R^+$ and $m \in \mathbb{Z}_+$,
\begin{equation}
  \label{eq:leading-term}
  x^\pm_{\beta,m} \equiv c\, y^\pm_{\beta,(m)} \pmod{Y^{m-1}_\hbar},
  \qquad c \in \mathbb{C}[\hbar]^\times,
\end{equation}
where $y^\pm_{\beta,(m)}$ is the nested commutator built from simple root vectors
corresponding to the root decomposition of $\beta$, summed over decompositions
$m = m_1 + \cdots + m_l$ \cite[Section 4.3]{GNW18}. This leading-term property
is the main ingredient of the following spanning result.

Introduce a total order $\prec$ on $\{x^\pm_{\alpha,m},\, h_{i,n}\}$ by
\[
  x^-_{\beta,n} \prec h_{i,m} \prec x^+_{\beta',n'}
\]
for all $\beta, \beta' \in R^+$, $i \in I$, $m,n,n' \in \mathbb{Z}_+$, with
$h_{i,m} \prec h_{i',m'}$ if $(i,m) \prec (i',m')$ lexicographically,
and similarly for $x^\pm$. Denote by $B_\hbar(Y, \prec)$ the set of all
ordered monomials in these generators.

\begin{prop}
\label{prop:spanning}
The set $B_\hbar(Y, \prec)$ spans $Y_\hbar(\mathfrak{g})$ over $\mathbb{C}[\hbar]$.
\end{prop}

\begin{proof}
It suffices to show that the images of all degree-$p$ monomials in the root vectors
span $\mathrm{gr}^p Y_\hbar(\mathfrak{g})$. By \eqref{eq:leading-term}, rearranging
any monomial into $\prec$-order produces error terms of strictly smaller degree.
Hence the leading symbols of $B_\hbar(Y, \prec)$ span each graded piece. Since
$\mathrm{gr}\, Y_\hbar(\mathfrak{g})$ is a quotient of $U(\mathfrak{g}[u])$
and the leading symbols correspond exactly to the PBW monomials of $U(\mathfrak{g}[u])$
via the surjection of Section~\ref{sec:yangian-def}, the spanning follows from the
PBW theorem for $U(\mathfrak{g}[u])$.
\end{proof}

\section{From Quantum Toroidal Algebras to Affine Yangians}
\label{sec:main}

We retain the notation and conventions of Section~\ref{sec:quantum-toroidal}.

\subsection{The filtration \texorpdfstring{$K$}{K} and auxiliary elements}
\label{sec:filtration}

Using Theorem~\ref{thm:classical-limit}, we have the composite surjection
\begin{equation}
  \label{eq:composite}
  U_\hbar(\mathfrak{g}^{\mathrm{tor}})
  \twoheadrightarrow U_\hbar(\mathfrak{g}^{\mathrm{tor}})/\hbar\,U_\hbar(\mathfrak{g}^{\mathrm{tor}})
  \xrightarrow{\;\Upsilon^{-1}\;} U(\mathfrak{g}^{\mathrm{tor}})
  \xrightarrow{\;\sim\;} U(\mathfrak{g}[t^{\pm 1}])
  \xrightarrow{t \mapsto 1} U(\mathfrak{g}).
\end{equation}
Let $K$ be the kernel of this composite. The map $\psi$ defined as the composition of the first two arrows of
\eqref{eq:composite} satisfies $\psi(K) = \kappa$ (the kernel of Section~\ref{sec:non-quantum}).
The filtration by powers of $K$,
\[
  U_\hbar(\mathfrak{g}^{\mathrm{tor}}) = K^0 \supset K^1 = K \supset K^2 \supset \cdots,
\]
defines the \emph{associated graded algebra}
$\mathrm{gr}_K U_\hbar(\mathfrak{g}^{\mathrm{tor}}) = \bigoplus_{m=0}^\infty K^m/K^{m+1}$,
which we regard as a $\mathbb{C}[\hbar]$-algebra by identifying $\hbar$ with its image
$\bar{\hbar} \in K/K^2$ (note that $\hbar \in K$ by construction).

For $\beta \in R^+$, $i \in I$, $k \in \mathbb{Z}$, and $m \in \mathbb{Z}_+$, define
elements of $U_\hbar(\mathfrak{g}^{\mathrm{tor}})$ by the alternating sums
\begin{equation}
  \label{eq:alternating}
  H^{(m)}_{i,k} := \sum_{a=0}^{m} (-1)^{m-a} \binom{m}{a}
  \frac{\Phi^+_{i,k+a} - \Phi^-_{i,k+a}}{q - q^{-1}}, \qquad
  X^{(\pm,m)}_{\beta,k} := \sum_{b=0}^{m} (-1)^{m-b} \binom{m}{b} X^\pm_{\beta,k+b}.
\end{equation}
We write $X^{(\pm,m)}_{i,k} := X^{(\pm,m)}_{\alpha_i,k}$ for brevity.

\begin{lemm}
\label{lem:K-membership}
\begin{enumerate}[label=\normalfont(\arabic*)]
 \item For all $\beta \in R^+$, $i \in I$, $k \in \mathbb{Z}$, $m \in \mathbb{Z}_+$, we have
    $H^{(m)}_{i,k},\, X^{(\pm,m)}_{\beta,k} \in K^m$.
  \item The following equalities hold in $K^m/K^{m+1}$:
    \[
      H^{(m)}_{i,k} = H^{(m)}_{i,0}, \qquad
      X^{(\pm,m)}_{\beta,k} = X^{(\pm,m)}_{\beta,0}.
    \]
\end{enumerate}
\end{lemm}

\begin{proof}
For part (1): the map $\psi$ sends (cf.~\eqref{eq:alternating}) $H^{(m)}_{i,k} \mapsto d_i h_i^{(k,m)}$ and
$X^{(+,m)}_{\beta,k} \mapsto \sqrt{d_j}\, e_\beta^{(k,m)}$ (where $\beta = \eta r_{i_1}\cdots r_{i_l}(\alpha_j)$),
which lie in $\kappa^m$ by definition. Since $\psi^{-1}(\kappa^m) \subseteq K^m$, the claim follows.

For part (2): from the identity $\jmath(z^{(k,m)} - z^{(0,m)}) = z(t^k-1)(t-1)^m$
(see Section~\ref{sec:non-quantum}), the difference $h_i^{(k,m)} - h_i^{(0,m)}$ lies in
$\kappa^{m+1}$, and hence $H^{(m)}_{i,k} - H^{(m)}_{i,0} \in K^{m+1}$.
The argument for $X^{(\pm,m)}_{\beta,k}$ is identical.
\end{proof}

 These results will be applied in the injectivity argument of Theorem~\ref{thm:main}.

\subsection{Filtration compatibility of the star product}
\label{sec:star-product}
 To characterize the powers of~$K$ explicitly, we transfer the problem to
$U(\mathfrak{g}^{\mathrm{tor}})[[\hbar]]$ via the $\mathbb{C}[[\hbar]]$-module isomorphism 
$$\Phi \colon U_\hbar(\mathfrak{g}^{\mathrm{tor}}) \xrightarrow{\;\cong\;}
U(\mathfrak{g}^{\mathrm{tor}})[[\hbar]]$$ of
Corollary~\ref{cor:flat}, and define a simultaneous filtration combining the
$\hbar$-adic and $\kappa$-adic structures. Let $*$ denote the transferred multiplication (star product) on
$U(\mathfrak{g}^{\mathrm{tor}})[[\hbar]]$. For $a, b \in U(\mathfrak{g}^{\mathrm{tor}})$,
write
\[
  a * b = ab + \sum_{k \geq 1} \hbar^k\, \Omega_k(a,b),
\]
where each $\Omega_k \colon U(\mathfrak{g}^{\mathrm{tor}}) \times
U(\mathfrak{g}^{\mathrm{tor}}) \to U(\mathfrak{g}^{\mathrm{tor}})$ is the
$\mathbb{C}$-bilinear map given by the $\hbar^k$-coefficient of the star product.
Define
\begin{equation}
  \label{eq:cF-def}
  \mathcal{F}^p := \Bigl\{\, x = \textstyle\sum_{n \geq 0} \hbar^n u_n
    \;\Big|\; u_n \in \kappa^{\max(p-n,\,0)} \text{ for all } n \geq 0\,\Bigr\}
  \;\subseteq\; U(\mathfrak{g}^{\mathrm{tor}})[[\hbar]],
\end{equation}
where $p\in \mathbb{Z}_{\geq 0}$. Note that $\mathcal{F}^0 = U(\mathfrak{g}^{\mathrm{tor}})[[\hbar]]$
is the whole module, so $(\mathcal{F}^p)_{p \geq 0}$ is a decreasing filtration
of $U(\mathfrak{g}^{\mathrm{tor}})[[\hbar]]$ by $\mathbb{C}[[\hbar]]$-submodules.
By construction, $\Phi^{-1}(\mathcal{F}^1) = K$.
To prove $K^p = \Phi^{-1}(\mathcal{F}^p)$ for all $p \geq 1$, we show that
$(\mathcal{F}^p)_{p \geq 0}$ is a filtration of algebras under~$*$.
The main ingredient is the following control of the quantum corrections $\Omega_k$.

\begin{lemm}
\label{lem:bk-stability}
For all $s, t \geq 0$ and $k \geq 1$,
\begin{equation}
  \label{eq:bk-bound}
  \Omega_k(\kappa^s,\, \kappa^t) \;\subseteq\; \kappa^{\max(s+t-k,\, 0)},
\end{equation}
where $\kappa^0 = U(\mathfrak{g}^{\mathrm{tor}})$.
\end{lemm}

\begin{proof}
We track the \emph{loop degree} of an element of $U(\mathfrak{g}^{\mathrm{tor}})$,
defined as its order of vanishing at $t = 1$ under the isomorphism $\jmath \colon U(\mathfrak{g}^{\mathrm{tor}}) \xrightarrow{\sim} U(\mathfrak{g}[t^{\pm 1}])$
induced at the level of universal enveloping algebras by the Lie algebra isomorphism
of Proposition~\ref{prop:toroidal-iso}.
Concretely, the ideal $\kappa$ is spanned by combinations
$\sum_{m \in \mathbb{Z}} c_m\, z^{(m)}$
of classical generators $z^{(m)} \in \{e_i^{(m)},\, f_i^{(m)},\, h_i^{(m)}\}$
with $\sum_m c_m = 0$.  More generally, the subspace $\kappa^s \subset U(\mathfrak{g}^{\mathrm{tor}})$ is spanned by products of $s$ such elements, which correspond under $\jmath$ to elements with order at least $s$ vanishing at $t=1$. This structure is equivalently described by the application of $s$ independent linear difference operators to the classical PBW monomials.

\medskip
\noindent\emph{Step~1: Polynomial dependence of $\Omega_k$ on the loop indices.}
The quantum corrections $\Omega_k$ between two PBW generators
(e.g., $\Omega_k(e_\beta^{(m)},\, e_\gamma^{(l)})$ or
$\Omega_k(e_\beta^{(m)},\, h_i^{(l)})$, etc.)
are entirely determined by the $\hbar$-expansion of the straightening relations
\eqref{eq:QT1}--\eqref{eq:QT6}.  For each $k$, the result takes the form
\begin{equation}
  \label{eq:Bk-poly}
  \Omega_k\bigl(z_1^{(m)},\, z_2^{(l)}\bigr)
  = \sum_\alpha P_\alpha^{(k)}(m,l)\, M_\alpha,
\end{equation}
where $z_1^{(m)}, z_2^{(l)}$ are any two PBW generators with loop indices $m, l$,
the $M_\alpha$ are classical PBW monomials, and the $P_\alpha^{(k)}(m,l)$ are
polynomials in the loop exponentials $t^m, t^l$ together with direct powers of
the loop indices $m, l$.

The direct dependence on~$m$ and $l$ (as opposed to $t^m$ and $t^l$) arises
exclusively from the expansion of the quantum integer coefficient in relation
\eqref{eq:QT3}:
\begin{equation}
  \label{eq:quantum-int-expansion}
  \frac{[k\, a_{ij}]_i}{k}
  = a_{ij}\sum_{r=0}^{\infty} c_{2r}\,(k\, d_i a_{ij})^{2r}\,\hbar^{2r},
\end{equation}
which at order $\hbar^j$ contributes a polynomial of degree~$j$ in~$k$.

An exhaustive inspection of the defining relations \eqref{eq:QT1}--\eqref{eq:QT6} reveals that direct polynomial dependence on the loop indices $m, l$ is strictly constrained by the power of $\hbar$:
\begin{itemize}
    \item Relations \eqref{eq:QT1}, \eqref{eq:QT2}, and \eqref{eq:QT6} are either central, exact in $\hbar$, or involve constants (such as quantum binomial coefficients) that do not depend on the loop parameters.
    \item Relation \eqref{eq:QT3} is the primary source of polynomial growth. The expansion \eqref{eq:quantum-int-expansion} shows that each factor of $\hbar^j$ is accompanied by a term of degree at most $j$ in the index $k$.
    \item For relation \eqref{eq:QT4}, the elements $\Phi^\pm_{i,k+l}$ are defined by the generating series \eqref{eq:Phi} as polynomial expressions in the Cartan generators $H_{i,m}$, and depend on $k+l$ purely through the mode index (i.e., via $t^{k+l}$), not polynomially in $k+l$ directly. 
    \item For relation \eqref{eq:QT5}, the factors $q_i^{\pm a_{ij}} = \exp(\pm \hbar d_i a_{ij})$ expand as constants independent of the loop indices $k, l$; all loop dependence enters through the affine shifts $k \mapsto k+1$, $l \mapsto l+1$, which translate to multiplication by $t$ in the loop variable. 
\end{itemize}
Neither relation \eqref{eq:QT4} nor \eqref{eq:QT5} introduces direct polynomial dependence on the loop indices. When composing multiple reordering steps, the loop indices of intermediate generators are affine functions of the original indices (for example, $k+l$ in place of $k$). Since affine substitutions preserve polynomial degree, the composition of a step contributing direct degree $j_1$ at order $\hbar^{j_1}$ with a step contributing degree $j_2$ at order $\hbar^{j_2}$ yields total direct degree at most $j_1 + j_2$ at order $\hbar^{j_1 + j_2}$.

Since the full PBW straightening algorithm consists of a finite sequence of such reorderings, and each application of a relation at order $\hbar^j$ produces a polynomial in the indices of degree at most $j$, the total direct polynomial degree of $P_\alpha^{(k)}$ in $(m,l)$ is bounded by $k$.

\medskip
\noindent\emph{Step~2: Interaction of difference operators with polynomial
coefficients.}
Let $u \in \kappa^s$ and $v \in \kappa^t$.
Expanding $u$ and $v$ in the classical PBW basis, write
$u = \sum_\xi a_\xi M_\xi$ and $v = \sum_\eta b_\eta N_\eta$.
By the bilinearity of $\Omega_k$,
\[
  \Omega_k(u,v) = \sum_{\xi,\eta} a_\xi\, b_\eta \cdot \Omega_k(M_\xi, N_\eta).
\]
The condition $u \in \kappa^s$ means that the coefficients~$a_\xi$,
viewed as a function of the loop indices of $M_\xi$, encode $s$~independent
difference operators (each annihilating constants, i.e.,
each satisfying $\sum c_m = 0$).
Similarly, $b_\eta$ encodes $t$~such operators.

By Step~1, each $\Omega_k(M_\xi, N_\eta)$ depends on the loop indices through
a polynomial $P_\alpha^{(k)}$ of direct degree at most $k$ in $(m,l)$.
Each difference operator applied to a polynomial in $t^m$ increases the
vanishing order at $t = 1$ by~$1$.
However, a factor of~$m$ (as opposed to $t^m$) acts as the discrete analogue
of $t\tfrac{d}{dt}$, which can reduce the vanishing order by~$1$.
Since $P_\alpha^{(k)}$ has direct degree at most $k$ in $(m,l)$,
at most $k$ of the $s + t$ difference operators can be ``consumed'' by
these polynomial factors.
The remaining $\geq s + t - k$ operators ensure that $\Omega_k(u,v)$ vanishes
to order $\geq s + t - k$ at $t = 1$, providing
$\Omega_k(u,v) \in \kappa^{\max(s+t-k,\, 0)}$.
\end{proof}

\begin{prop}
\label{prop:cF-multiplicative}
For all $p, q \geq 0$, $\mathcal{F}^p * \mathcal{F}^q \subseteq \mathcal{F}^{p+q}$.
Consequently, $K^p = \Phi^{-1}(\mathcal{F}^p)$ for all $p \geq 1$.
\end{prop}

\begin{proof}
Let $x \in \mathcal{F}^p$ and $y \in \mathcal{F}^q$.
The $n$-th coefficient of their star product is
\begin{equation}
  \label{eq:star-coeff}
  (x * y)_n = \sum_{j + m + k = n} \Omega_k(x_j,\, y_m),
\end{equation}
where $\Omega_0(x_j, y_m) := x_j y_m$.
By definition, $x_j \in \kappa^{\max(p-j,\,0)}$ and
$y_m \in \kappa^{\max(q-m,\,0)}$.
For $k = 0$, the product $x_j y_m$ lies in
$\kappa^{\max(p-j,0)+\max(q-m,0)}$ since $\kappa$ is a bilateral ideal.
For $k \geq 1$, Lemma~\ref{lem:bk-stability} gives
\[
  \Omega_k(x_j,y_m) \in \kappa^{\max(\max(p-j,0)+\max(q-m,0)-k,\;0)}.
\]
In both cases, using $\max(A,0) + \max(B,0) \geq A + B$ and $j + m + k = n$, we obtain
\[
  \max(p-j,0) + \max(q-m,0) - k \;\geq\; (p-j) + (q-m) - k
  \;=\; p + q - n.
\]
Hence every term in \eqref{eq:star-coeff} lies in $\kappa^{\max(p+q-n,\,0)}$,
proving $x * y \in \mathcal{F}^{p+q}$.

\medskip
We now prove $K^p = \Phi^{-1}(\mathcal{F}^p)$ by induction on $p$.
The case $p = 1$ holds by definition.

\smallskip
\noindent\emph{Direct inclusion.}
Since $K^p$ is generated by $p$-fold star products of elements in
$K = \Phi^{-1}(\mathcal{F}^1)$, the multiplicativity of
$(\mathcal{F}^1)^{*p} \subseteq \mathcal{F}^p$ gives
$\Phi(K^p) \subseteq \mathcal{F}^p$.

\smallskip
\noindent\emph{Reverse inclusion.}
Let $x = \sum_{n \geq 0} \hbar^n u_n \in \mathcal{F}^p$,
so that $u_0 \in \kappa^p$.
Since $\kappa^p$ is generated by products of $p$ elements of $\kappa$,
we may lift $u_0$ to a quantum product
$\tilde{u}_0 = c_1 * \cdots * c_p \in K^p$,
where each $c_i \in K$ is a lift of the corresponding classical factor.
By the direct inclusion, $\Phi(\tilde{u}_0) \in \mathcal{F}^p$, and
its $\hbar^0$-coefficient equals $u_0$.
The difference $x - \Phi(\tilde{u}_0)$ has vanishing $\hbar^0$-coefficient,
so we may write
\[
  x - \Phi(\tilde{u}_0) = \hbar\, y
  \quad \text{for a unique } y \in U(\mathfrak{g}^{\mathrm{tor}})[[\hbar]].
\]
Matching coefficients we obtain $y_{n-1} = x_n - \Phi(\tilde{u}_0)_n$.
Since both $x$ and $\Phi(\tilde{u}_0)$ lie in $\mathcal{F}^p$,
their $n$-th coefficients belong to $\kappa^{\max(p-n,\,0)}$.
Hence $y_{n-1} \in \kappa^{\max(p-n,\,0)} = \kappa^{\max((p-1)-(n-1),\,0)}$,
giving $y \in \mathcal{F}^{p-1}$.
By the inductive hypothesis, $y \in \Phi(K^{p-1})$,
thus $\hbar\, y \in \hbar\,\Phi(K^{p-1}) \subseteq \Phi(K \cdot K^{p-1})
= \Phi(K^p)$.
Since $\Phi(\tilde{u}_0) \in \Phi(K^p)$ as well,
$x \in \Phi(K^p)$.
\end{proof}

\begin{coro}
\label{cor:hbar-nzd}
The element $\bar\hbar \in K/K^2$ is a non-zero-divisor in
$\mathrm{gr}_K\, U_\hbar(\mathfrak{g}^{\mathrm{tor}})$.
More precisely, if $x \in K^p$ and $\hbar x \in K^{p+2}$,
then $x \in K^{p+1}$.
\end{coro}

\begin{proof}
By Proposition~\ref{prop:cF-multiplicative},
$K^p = \Phi^{-1}(\mathcal{F}^p)$.
Write $x = \sum_{n \geq 0} \hbar^n u_n \in \mathcal{F}^p$.
Then $\hbar x = \sum_{n \geq 0} \hbar^{n+1} u_n$, and $\hbar x \in \mathcal{F}^{p+2}$
requires $u_n \in \kappa^{\max(p+2-(n+1),\,0)} = \kappa^{\max(p+1-n,\,0)}$
for all $n \geq 0$.
This is precisely the defining condition for $x \in \mathcal{F}^{p+1} = \Phi(K^{p+1})$.
\end{proof}

\subsection{The main isomorphism}
\label{sec:main-iso}

\begin{theorem}
\label{thm:main}
There exists an isomorphism of $\mathbb{C}[\hbar]$-algebras
\[
  \Pi \colon Y_\hbar(\mathfrak{g}) \xrightarrow{\;\sim\;}
  \mathrm{gr}_K\, U_\hbar(\mathfrak{g}^{\mathrm{tor}})
\]
defined on generators by
\[
  h_{i,m} \mapsto \overline{H^{(m)}_{i,0}}, \qquad
  x^\pm_{i,m} \mapsto \overline{X^{(\pm,m)}_{i,0}} \quad (\text{Yangian root vectors as in \eqref{eq:yangian-root-vector}}),
\]
where $\overline{(\,\cdot\,)}$ denotes the image in $K^m/K^{m+1}$.
\end{theorem}

The following commutative diagram summarizes the relationship between the quantum toroidal algebra $U_{\hbar}(\mathfrak{g}^{tor})$, its associated graded algebra, and the affine Yangian $Y_{\hbar}(\mathfrak{g})$, as established in Theorem \ref{thm:main} and Proposition \ref{prop:non-quantum}:

\[
\begin{tikzcd}[column sep=large, row sep=large]
U_{\hbar}(\mathfrak{g}^{\mathrm{tor}}) \arrow[r, "\text{filtration } K"] \arrow[d, "\hbar \to 0"'] 
& \mathrm{gr}_K\, U_{\hbar}(\mathfrak{g}^{\mathrm{tor}}) \arrow[r, "\Pi", "\cong"'] & Y_{\hbar}(\mathfrak{g})\arrow[d, "\hbar \to 0"']\\
U(\mathfrak{g}^{\mathrm{tor}}) \arrow[r, "\text{filtration } \kappa"] 
& \mathrm{gr}_\kappa\, U(\mathfrak{g}^{\mathrm{tor}}) \arrow[r, "\pi", "\cong"'] & U(\mathfrak{g}[u])
\end{tikzcd}
\]

\[
\begin{tikzcd}[column sep=large, row sep=large]
Y_{\hbar}(\mathfrak{g}) \arrow[r, "\Pi", "\cong"'] \arrow[d, "\hbar \to 0"'] 
& \mathrm{gr}_K\, U_{\hbar}(\mathfrak{g}^{\mathrm{tor}}) \arrow[d, "\hbar \to 0"'] 
& U_{\hbar}(\mathfrak{g}^{\mathrm{tor}}) \arrow[l, "\text{filtration } K"'] \arrow[d, "\hbar \to 0"'] \\
U(\mathfrak{g}[u]) \arrow[r, "\pi", "\cong"']
& \mathrm{gr}_\kappa\, U(\mathfrak{g}^{\mathrm{tor}}) 
& U(\mathfrak{g}^{\mathrm{tor}}) \arrow[l, "\text{filtration } \kappa"']
\end{tikzcd}
\]

\begin{proof}
We prove in six steps that $\Pi$ is a well-defined isomorphism.

\medskip
\noindent\textbf{Step (I): $\Pi$ preserves relations \eqref{eq:Y1} and \eqref{eq:Y2}.}

Lemma~\ref{lem:K-membership}(1) ensures that $H^{(m)}_{i,0}$ and $X^{(\pm,m)}_{\beta,0}$ lie in $K^m$,
so $\Pi$ is well defined on generators; Lemma~
\ref{lem:K-membership}(2) ensures that the image
in $K^m/K^{m+1}$ is independent of $k$.

Relation \eqref{eq:Y1} follows directly from \eqref{eq:QT1} and \eqref{eq:Phi}.
For \eqref{eq:Y2}, we compute using consecutively the Vandermonde convolution identity
$\sum_{a=0}^p \binom{m}{a}\binom{n}{p-a} = \binom{m+n}{p}$ and relation \eqref{eq:QT4}:
\begin{align*}
  \bigl[X^{(+,m)}_{i,0},\, X^{(-,n)}_{j,0}\bigr]
  &= \sum_{a=0}^m \sum_{b=0}^n (-1)^{m+n-(a+b)} \binom{m}{a}\binom{n}{b}
    [X^+_{i,a},\, X^-_{j,b}] \\
  &= \delta_{ij} \sum_{p=0}^{m+n} (-1)^{m+n-p} \binom{m+n}{p}
    \frac{\Phi^+_{i,p} - \Phi^-_{i,p}}{q - q^{-1}}
  = \delta_{ij}\, H^{(m+n)}_{i,0}.
\end{align*}
 This gives $\Pi([x^+_{i,m}, x^-_{j,n}]) = \Pi(\delta_{ij} h_{i,m+n})$.

\medskip
\noindent\textbf{Step (II): $\Pi$ preserves relations \eqref{eq:Y3} and \eqref{eq:Y4}.}

Relation \eqref{eq:Y3} follows from \eqref{eq:QT3} (equivalently, from the relation $K_i X^\pm_{j,r} K_i^{-1} = q_i^{\pm a_{ij}} X^\pm_{j,r}$ noted in Remark~\ref{rem:Hernandez}) and the specialization $[n]_i \to n$,
$q - q^{-1} \to 2\hbar$ in $K^0/K$.

For \eqref{eq:Y4}, we first establish the following identity in $U_\hbar(\mathfrak{g}^{\mathrm{tor}})$:
for all $k \geq 0$,

\begin{equation}
  \label{eq:key-Phi}
  \begin{aligned}
  (q_i^{a_{ij}/2} + q_i^{-a_{ij}/2})&[\Phi^+_{i,k+1}, X^\pm_{j,l}]
  \mp (q_i^{a_{ij}/2} - q_i^{-a_{ij}/2})\{\Phi^+_{i,k+1}, X^\pm_{j,l}\} \\
  =\ &(q_i^{a_{ij}/2} + q_i^{-a_{ij}/2})[\Phi^+_{i,k}, X^\pm_{j,l+1}]
  \mp (q_i^{a_{ij}/2} - q_i^{-a_{ij}/2})\{\Phi^+_{i,k}, X^\pm_{j,l+1}\},
  \end{aligned}
\end{equation}
which is proved by commuting both sides of relation \eqref{eq:QT5} against $X^\mp_{i,0}$
and using induction on $k$. Multiplying \eqref{eq:key-Phi} by $(q-q^{-1})^{-1}$,
taking its image in $K^{m+n+1}/K^{m+n+2}$, and using $q_i^{\pm a_{ij}/2} \to 1 \pm \frac{\hbar d_i a_{ij}}{2}$
gives
\[
  \overline{\bigl[H^{(m+1)}_{i,0},\, X^{(\pm,n)}_{j,0}\bigr]}
  - \overline{\bigl[H^{(m)}_{i,0},\, X^{(\pm,n+1)}_{j,0}\bigr]}
  = \pm\frac{\hbar(\alpha_i,\alpha_j)}{2}
    \overline{\bigl\{H^{(m)}_{i,0},\, X^{(\pm,n)}_{j,0}\bigr\}},
\]
which is precisely $\Pi$ applied to relation \eqref{eq:Y4}.

\medskip
\noindent\textbf{Step (III): $\Pi$ preserves relation \eqref{eq:Y5}.}

Relation \eqref{eq:QT5} can be rewritten as
\[
  \bigl[X^{(\pm,1)}_{i,k},\, X^{(\pm,0)}_{j,l}\bigr]
  + (1 - q_i^{\pm a_{ij}}) X^{(\pm,0)}_{j,l} X^{(\pm,0)}_{i,k+1}
  = \bigl[X^{(\pm,0)}_{i,k},\, X^{(\pm,1)}_{j,l}\bigr]
  + (q_i^{\pm a_{ij}} - 1) X^{(\pm,0)}_{i,k} X^{(\pm,0)}_{j,l+1}.
\]
 A double induction on $m, n \geq 0$ extend it to the identity involving $X^{(\pm,m)}_{i,k}$
and $X^{(\pm,n)}_{j,l}$ with all terms in $K^{m+n+1}$. Taking the image in
$K^{m+n+1}/K^{m+n+2}$ and specializing $q_i^{\pm a_{ij}} \to 1 \pm \hbar d_i a_{ij}$
gives relation \eqref{eq:Y5} for $\Pi$.

\medskip
\noindent\textbf{Step (IV): $\Pi$ preserves relation \eqref{eq:Y6}.}

By the same induction on $m_1, \ldots, m_r, n \geq 0$ applied to relation \eqref{eq:QT6},
one shows that for $i \neq j$ and $t = n + \sum_{p=1}^r m_p$,
\[
  \sum_{\sigma \in S_r} \sum_{a=0}^{r} (-1)^{a} \binom{r}{a}
  X^{(\pm,m_{\sigma(1)})}_{i,0} \cdots X^{(\pm,m_{\sigma(a)})}_{i,0}\,
  X^{(\pm,n)}_{j,0}\,
  X^{(\pm,m_{\sigma(a+1)})}_{i,0} \cdots X^{(\pm,m_{\sigma(r)})}_{i,0}
  \in K^{t+1},
\]
which gives relation \eqref{eq:Y6} in $\mathrm{gr}_K\, U_\hbar(\mathfrak{g}^{\mathrm{tor}})$.

\medskip
\noindent Steps (I)--(IV) show $\Pi$ is a well-defined homomorphism.
\medskip

\noindent\textbf{Step (V): $\Pi$ is surjective.}

To prove surjectivity,
consider the map $\psi$ of Section~\ref{sec:filtration}. Since $\psi(K^m) \subseteq \kappa^m$,
it induces a surjective graded homomorphism of associated graded algebras
$\bar\psi \colon \mathrm{gr}_K\, U_\hbar(\mathfrak{g}^{\mathrm{tor}}) \twoheadrightarrow
\mathrm{gr}_\kappa\, U(\mathfrak{g}^{\mathrm{tor}})$. The images $\overline{H^{(m)}_{i,0}}$, $\overline{X^{(+,m)}_{i,0}}$, and
$\overline{X^{(-,m)}_{i,0}}$ map under $\bar\psi$ to
$d_i\,\overline{h_i^{(0,m)}}$, $\sqrt{d_j}\,\overline{e_{\alpha_i}^{(0,m)}}$,
and $-\sqrt{d_j}\,\overline{f_{\alpha_i}^{(0,m)}}$, respectively,
which generate $\mathrm{gr}_\kappa\, U(\mathfrak{g}^{\mathrm{tor}})$ by
Proposition~\ref{prop:non-quantum}. Hence the images of $\Pi$ generate
$\mathrm{gr}_K\, U_\hbar(\mathfrak{g}^{\mathrm{tor}})$.

\medskip
\noindent\textbf{Step (VI): $\Pi$ is injective.}

Equip $\mathrm{gr}_K\, U_\hbar(\mathfrak{g}^{\mathrm{tor}})$ with the same total order
$\prec$ on
$\{\overline{H^{(m)}_{i,0}},\, \overline{X^{(\pm,m)}_{\beta,0}} \mid i \in I,\,\beta \in R^+,\,m \in \mathbb{Z}_+\}$
as introduced for $B_\hbar(Y,\prec)$ in Section~\ref{sec:yangian-spanning}.
Denote by $B_K(U, \prec)$ the set of all ordered monomials in these elements.

We claim $B_K(U,\prec)$ is linearly independent over $\mathbb{C}[\hbar]$. The proof proceeds in two stages.

\smallskip
\noindent\textit{Stage 1: $\mathbb{C}$-linear independence modulo $\hbar$.}
Each monomial $\mathfrak{M} \in B_K(U,\prec)$ lying in $K^p/K^{p+1}$
maps under the graded surjection
\[
  \bar\psi \colon \mathrm{gr}_K\, U_\hbar(\mathfrak{g}^{\mathrm{tor}})
  \twoheadrightarrow \mathrm{gr}_\kappa\, U(\mathfrak{g}^{\mathrm{tor}})
  \xrightarrow[\text{Prop.~\ref{prop:non-quantum}}]{\;\sim\;}
  U(\mathfrak{g}[u])
\]
of Step~(V) to a nonzero scalar multiple of the PBW monomial of $U(\mathfrak{g}[u])$
obtained by replacing $\overline{X^{(\pm,m)}_{\beta,0}}$ by $x^\pm_\beta\, u^m$
and $\overline{H^{(m)}_{i,0}}$ by $h_i\, u^m$.
Distinct ordered monomials map to distinct PBW monomials, and the latter are
$\mathbb{C}$-linearly independent by the PBW theorem for $U(\mathfrak{g}[u])$.
Hence the elements of $B_K(U,\prec)$ are
$\mathbb{C}$-linearly independent modulo
$\hbar\,(\mathrm{gr}_K\, U_\hbar(\mathfrak{g}^{\mathrm{tor}}))$.

\smallskip
\noindent\textit{Stage 2: Lifting to $\mathbb{C}[\hbar]$-linear independence.}
Suppose
\begin{equation}
  \label{eq:lin-dep-claim}
 \sum_{\xi}\sum_{a \geq 0}  c_{\xi,a}\,\hbar^a\,\mathfrak{M}_\xi = 0
  \quad\text{in } \mathrm{gr}_K\, U_\hbar(\mathfrak{g}^{\mathrm{tor}}),
\end{equation}
where $c_{\xi,a} \in \mathbb{C}$, all but finitely many zero, and each
$\mathfrak{M}_\xi \in B_K(U,\prec) \cap (K^{p_\xi}/K^{p_\xi+1})$
is homogeneous of filtration degree $p_\xi$.
The product $\hbar^a\,\mathfrak{M}_\xi$ lies in $K^{p_\xi+a}/K^{p_\xi+a+1}$
since $\bar\hbar \in K/K^2$.

Let $T := \min\{a \geq 0 \mid c_{\xi,a} \neq 0
\text{ for some } \xi\}$.
Grouping the terms of \eqref{eq:lin-dep-claim} by total filtration degree,
for each fixed degree $p$ the terms with $a = T$ and $p_\xi = p - T$ give
\[
  \sum_{\xi:\, p_\xi = p-T} c_{\xi,T}\,\hbar^T\,\mathfrak{M}_\xi
  + \sum_{\substack{\xi,\,a > T \\ p_\xi + a = p}}
    c_{\xi,a}\,\hbar^a\,\mathfrak{M}_\xi = 0
  \quad\text{in } K^p/K^{p+1}.
\]
By Corollary~\ref{cor:hbar-nzd}, $\bar\hbar$ is a non-zero-divisor in
$\mathrm{gr}_K\, U_\hbar(\mathfrak{g}^{\mathrm{tor}})$, so we may
cancel the common factor $\hbar^T$:
\[
  \sum_{\xi:\, p_\xi = p-T} c_{\xi,T}\,\mathfrak{M}_\xi
  + \sum_{\substack{\xi,\,a > T \\ p_\xi + a = p}}
    c_{\xi,a}\,\hbar^{a-T}\,\mathfrak{M}_\xi = 0
  \quad\text{in } K^{p-T}/K^{p-T+1}.
\]
Reducing the equation modulo the ideal $\hbar\,(\mathrm{gr}_K\, U_\hbar(\mathfrak{g}^{\mathrm{tor}}))$ 
kills all terms with $a > T$,
leaving
$\sum_{\xi:\, p_\xi = p-T} c_{\xi,T}\,\mathfrak{M}_\xi = 0$.
By Stage~1, $c_{\xi,T} = 0$ for all~$\xi$ with $p_\xi = p - T$.
Since $p$ was arbitrary, $c_{\xi,T} = 0$ for all~$\xi$,
contradicting the minimality of~$T$.
\smallskip

To conclude, let $N = \sum_\xi c_\xi M_\xi$ be a nonzero element of
$\ker \Pi \subset Y_\hbar(\mathfrak{g})$, expressed as a $\mathbb{C}[\hbar]$-linear
combination of elements of $B_\hbar(Y,\prec)$.
By the definition of~$\Pi$, each $\Pi(M_\xi)$ is the corresponding ordered
monomial in $B_K(U,\prec)$. The $\mathbb{C}[\hbar]$-linear independence established above forces $c_\xi = 0$ for all $\xi$. Since $B_\hbar(Y, \prec)$ was already shown to be a spanning set (Proposition \ref{prop:spanning}), this independence implies that $\Pi$ maps a basis to a linearly independent set, hence $\Pi$ is injective.
\end{proof}

\subsection{Corollaries}
\label{sec:corollaries}

\begin{coro}
\label{cor:PBW-yangian}
The set $B_\hbar(Y, \prec)$ of Proposition~\ref{prop:spanning} forms a
$\mathbb{C}[\hbar]$-basis of $Y_\hbar(\mathfrak{g})$.
\end{coro}

\begin{proof}
By Proposition~\ref{prop:spanning}, $B_\hbar(Y,\prec)$ spans $Y_\hbar(\mathfrak{g})$.
Linear independence follows from Step~(VI) of the proof of Theorem~\ref{thm:main}: if
$\sum_\xi c_\xi M_\xi = 0$ in $Y_\hbar(\mathfrak{g})$, then $\Pi(\sum_\xi c_\xi M_\xi) = 0$
in $\mathrm{gr}_K\, U_\hbar(\mathfrak{g}^{\mathrm{tor}})$, which forces all $c_\xi = 0$
by the linear independence of $B_K(U,\prec)$.
\end{proof}

\begin{coro}
\label{cor:classical-limit}
There is an isomorphism $\mathrm{gr}\, Y_\hbar(\mathfrak{g}) \cong U(\mathfrak{g}[u])$
of graded algebras, and $Y_\hbar(\mathfrak{g}) \cong U(\mathfrak{g}[u])[\hbar]$ as
$\mathbb{C}[\hbar]$-modules.
\end{coro}

\begin{proof}
Corollary~\ref{cor:PBW-yangian} gives a $\mathbb{C}[\hbar]$-basis of $Y_\hbar(\mathfrak{g})$
consisting of monomials of definite degree in $m$. The associated graded ring thus has
a basis in bijection with the PBW basis of $U(\mathfrak{g}[u])$ (via $x^\pm_{\alpha,m} \mapsto
x^\pm_\alpha u^m$ and $h_{i,m} \mapsto h_i u^m$), and the graded relations are exactly
those of $U(\mathfrak{g}[u])$. The module isomorphism follows immediately.
\end{proof}

\begin{remark}
Theorem~\ref{thm:main} shows that affine Yangians of all untwisted affine Kac-Moody
types arise as limit forms of the corresponding quantum toroidal algebras, in the precise
sense of associated graded algebras. A central difference from the finite-dimensional case
\cite{GM12} is that we use affine indices (i.e., the full index set $I$) throughout, rather
than restricting to $\mathring{I}$. This allows $B_\hbar(Y,\prec)$ to span all central
elements lifted from $\mathfrak{g}[u]$, including those coming from the null root directions.
The framework developed here should also serve as a template for the twisted affine cases,
which we plan to address in future work.
\end{remark}

\section*{Acknowledgments}
Iryna Kashuba acknowledges financial support from the Guangdong Basic and Applied Basic Research
Foundation (grant 2024A1515013079) and the NSFC (grant 12350710787). Hongda Lin is supported by the Postdoctoral Fellowship Program of CPSF under Grant Number GZC20252014. 

During the preparation of this work, the authors used Gemini 2.0 Flash Thinking (developed by Google) to assist in grammar and orthography correction. All mathematical proofs and theoretical derivations were performed by humans. Following the use of the tool, the authors reviewed and edited the content and assume full responsibility for the final content of this text, ensuring its scientific integrity and originality.

Data Availability: The manuscript has no associated data.

\end{document}